\documentclass[12pt,british]{article}
\usepackage{textcomp}
\usepackage[latin9]{inputenc}
\usepackage{babel}
\usepackage{refstyle}
\usepackage{mathtools}
\usepackage{amsmath}
\usepackage{amsthm}
\usepackage{amssymb}
\usepackage{stmaryrd}
\usepackage{wasysym}
\usepackage[all]{xy}
\usepackage[bookmarks=true,bookmarksnumbered=true,bookmarksopen=true,bookmarksopenlevel=1,
 breaklinks=false,pdfborder={0 0 0},pdfborderstyle={},backref=false,colorlinks=false]
 {hyperref}
\hypersetup{pdftitle={GGT},
 pdfauthor={Tal Cohen and Itamar Vigdorovich},
 pdfstartview={FitH}}

\makeatletter

\AtBeginDocument{\providecommand\propref[1]{\ref{prop:#1}}}
\AtBeginDocument{\providecommand\secref[1]{\ref{sec:#1}}}
\AtBeginDocument{\providecommand\thmref[1]{\ref{thm:#1}}}
\AtBeginDocument{\providecommand\defref[1]{\ref{def:#1}}}
\AtBeginDocument{\providecommand\lemref[1]{\ref{lem:#1}}}
\AtBeginDocument{\providecommand\corref[1]{\ref{cor:#1}}}
\newcommand*\LyXoverline{\raisebox{2.6ex}{\_}}
\newcommand{\lyxmathsym}[1]{\ifmmode\begingroup\def\b@ld{bold}
  \text{\ifx\math@version\b@ld\bfseries\fi#1}\endgroup\else#1\fi}

\DeclareTextSymbolDefault{\textquotedbl}{T1}
\RS@ifundefined{subsecref}
  {\newref{subsec}{name = \RSsectxt}}
  {}
\RS@ifundefined{thmref}
  {\def\RSthmtxt{theorem~}\newref{thm}{name = \RSthmtxt}}
  {}
\RS@ifundefined{lemref}
  {\def\RSlemtxt{lemma~}\newref{lem}{name = \RSlemtxt}}
  {}

\theoremstyle{plain}
\newtheorem{thm}{\protect\theoremname}[section]
\theoremstyle{definition}
\newtheorem{defn}[thm]{\protect\definitionname}
\theoremstyle{plain}
\newtheorem{cor}[thm]{\protect\corollaryname}
\theoremstyle{plain}
\newtheorem{question}[thm]{\protect\questionname}
\theoremstyle{plain}
\newtheorem{lem}[thm]{\protect\lemmaname}
\theoremstyle{plain}
\newtheorem{prop}[thm]{\protect\propositionname}
\theoremstyle{remark}
\newtheorem*{acknowledgement*}{\protect\acknowledgementname}
\theoremstyle{remark}
\newtheorem{notation}[thm]{\protect\notationname}
\theoremstyle{plain}
\newtheorem*{fact*}{\protect\factname}
\theoremstyle{remark}
\newtheorem{rem}[thm]{\protect\remarkname}

\usepackage{fullpage}
\date{}

\newref{lem}{refcmd={Lemma \ref{#1}}}
\newref{cor}{refcmd={Corollary \ref{#1}}}
\newref{prop}{refcmd={Proposition \ref{#1}}}
\newref{exa}{refcmd={Example \ref{#1}}}
\newref{thm}{refcmd={Theorem \ref{#1}}}
\newref{def}{refcmd={Definition \ref{#1}}}
\newref{sec}{refcmd={Section \ref{#1}}}

\def\moverlay{\mathpalette\mov@rlay}
\def\mov@rlay#1#2{\leavevmode\vtop{%
   \baselineskip\z@skip \lineskiplimit-\maxdimen
   \ialign{\hfil$\m@th#1##$\hfil\cr#2\crcr}}}
\newcommand{\charfusion}[3][\mathord]{
    #1{\ifx#1\mathop\vphantom{#2}\fi
        \mathpalette\mov@rlay{#2\cr#3}
      }
    \ifx#1\mathop\expandafter\displaylimits\fi}
\makeatother

\makeatother

\providecommand{\acknowledgementname}{Acknowledgement}
\providecommand{\corollaryname}{Corollary}
\providecommand{\definitionname}{Definition}
\providecommand{\factname}{Fact}
\providecommand{\lemmaname}{Lemma}
\providecommand{\notationname}{Notation}
\providecommand{\propositionname}{Proposition}
\providecommand{\questionname}{Question}
\providecommand{\remarkname}{Remark}
\providecommand{\theoremname}{Theorem}

\begin{document}
\title{Lifting Generators in Connected Lie Groups}
\author{Tal Cohen\thanks{Weizmann Institute of Science, tal.cohen@weizmann.ac.il.}
~and Itamar Vigdorovich\thanks{University of California San Diego, ivigdorovich@ucsd.edu.}\\
\emph{}\\
\emph{In memory of Gennady A. Noskov}\\
}

\maketitle
\global\long\def\GasC#1#2{\text{\lightning}_{#1}(#2)}%

\global\long\def\gas#1{\text{\lightning}^{*}(#1)}%

\global\long\def\Gas#1{\text{\lightning}(#1)}%

\global\long\def\act{\curvearrowright}%

\global\long\def\res{\!\restriction}%

\global\long\def\tr#1{\mathrm{d}(#1)}%

\global\long\def\sr#1{\rho(#1)}%

\global\long\def\conn#1{#1^{0}}%

\global\long\def\explain#1#2{\underset{\underset{\mathclap{#2}}{\downarrow}}{#1}}%

\global\long\def\bbm#1{\boldsymbol{#1}}%

\global\long\def\N{\mathbb{N}}%
\global\long\def\Z{\mathbb{Z}}%
\global\long\def\Q{\mathbb{Q}}%
\global\long\def\R{\mathbb{R}}%
\global\long\def\C{\mathbb{C}}%
\global\long\def\H{\mathbb{H}}%
\global\long\def\T{\mathbb{T}}%
\global\long\def\P{\mathbb{P}}%

\begin{abstract}
Given an epimorphism between topological groups $f:G\to H$, when
can a generating set of $H$ be lifted to a generating set of $G$? 

We show that for connected Lie groups the problem is fundamentally
abelian: generators can be lifted if and only if they can be lifted
in the induced map between the abelianisations (assuming the number
of generators is at least the minimal number of generators of $G$).
As a consequence, we deduce that connected perfect Lie groups satisfy
the Gasch\"utz lemma: generating sets of quotients can always be lifted.
If the Lie group is not perfect, this may fail. The extent to which
a group fails to satisfy the Gasch\"utz lemma is measured by its\emph{
Gasch\"utz rank}, which we bound for all connected Lie groups, and compute
exactly in most cases. Additionally, we compute the maximal size of
an irredundant generating set of connected abelian Lie groups, and
discuss connections between such generation problems with the Wiegold
conjecture. 

\end{abstract}

\section{Introduction}

Many problems and results in group theory are concerned with generating
sets of groups, and with lifting generators through quotients. A classical
example is the Gasch\"utz lemma, which states the following: given an
epimorphism between finite groups $f:G\to H$, every generating set
$h_{1},...,h_{n}$ of $H$ can be lifted via $f$ to a generating
set $g_{1},...,g_{n}$ of $G$, provided that $G$ can be generated
by $n$ elements. This result has proved to be very useful in the
study of finite groups and the automorphism group of the free group
$\mathrm{Aut}(F_{n})$. Furthermore, its extension to profinite groups
serves as a key tool in the study of presentations of groups, subgroup
growth, and the Wiegold conjecture (which we further discuss below),
see \cite{Lub01,LS12,GL09,BCL22}.

Generator lifting problems may be studied for general topological
groups. In such case, generation is understood in the topological
sense: a subset $S$ of a topological group $G$ is said to \emph{generate}
$G$ if the group abstractly generated by $S$ (which is denoted by
$\left\langle S\right\rangle $) is dense in $G$. We denote by $\tr G$
the minimal size of a generating set of $G$. 

Henceforward, we assume all homomorphisms between topological groups
are continuous. By an epimorphism, we mean a surjective (continuous)
homomorphism. It will be useful \textit{not }to assume by default
that topological groups are Hausdorff; rather, we will explicitly
indicate whenever the Hausdorff property is required. 

Following \cite{CG18}, we define:
\begin{defn}
\label{def:Gaschutz-group}A topological group $G$ is said to be
\emph{Gasch\"utz} if, whenever $f:G\to H$ is an open\footnote{The restriction to open maps is important when dealing with topological
groups which admit weaker group topologies. For more on weak topologies,
see \cite{DS91,BL22}.} epimorphism onto a Hausdorff topological group $H$, and $h_{1},...,h_{n}\in H$
generate $H$ with $n\geqslant\tr G$, there exist lifts $g_{1},...,g_{n}\in G$
(i.e., elements satisfying $f(g_{i})=h_{i}$) that generate $G$.
\end{defn}

In the terminology of Definition \ref{def:Gaschutz-group}, the Gasch\"utz
lemma states that finite groups are Gasch\"utz \cite{Gas55}. In \cite{CG18},
it is shown that all first-countable\footnote{It is in fact enough to assume the connected component of the identity
is first-countable.} compact groups are Gasch\"utz as well. Our first result welcomes another
class of groups to this family:
\begin{thm}
\label{thm:intro-perfect-groups}Every perfect connected Lie group
is Gasch\"utz.
\end{thm}

This includes, of course, all connected semisimple Lie groups. The
assumption that $G$ is perfect (i.e., that the derived subgroup $G'$
of $G$ is dense) is essential. In fact, any connected abelian Lie
group which is not compact and not isomorphic to $\R$ is not Gasch\"utz
(see \propref{lowerbnd}). We demonstrate that this is the only source
of difficulty, in the sense that the challenge of lifting generators
is ultimately reduced to the abelian case. This is the content of
our main theorem, from which Theorem \ref{thm:intro-perfect-groups}
is an immediate consequence.
\begin{thm}
\label{thm:intro-main}Let $f:G\to H$ an open epimorphism between
connected Lie groups. Let $f^{\mathrm{ab}}:G/\overline{G'}\to H/f\left(\overline{G'}\right)$
be the map\footnote{Observe that $f\left(\overline{G'}\right)$ might not be closed in
$H$, so $H/f\left(\overline{G'}\right)$ might not be Hausdorff.} induced by $f$. If $h_{1},\dots,h_{n}\in H$ generate $H$ with
$n\geqslant\tr G$, then they admit generating lifts to $G$ via $f$
if and only if $h_{1}f\left(\overline{G'}\right),\dots,h_{n}f\left(\overline{G'}\right)$
admit generating lifts to $G/\overline{G'}$ via $f^{\mathrm{ab}}$.
\end{thm}

We are thus left with understanding the abelian case. As mentioned
above, even $\mathbb{R}\times\mathbb{R}$ is not Gasch\"utz: it is generated
by $3$ elements but the quotient $\mathbb{R}\times\mathbb{R}/\mathbb{Z}$
admits $3$ generators that cannot be lifted. However, every $4$
generators can be lifted. We therefore introduce the following notion.
\begin{defn}
\label{def:gaschutz-rank}Let $f:G\to H$ be an epimorphism between
topological groups and let $n\in\mathbb{N}$. We say that $f$ is
$n$-Gasch\"utz if for every $n$ elements $h_{1},...,h_{n}\in H$ that
generate $H$ there are lifts $g_{1},...,g_{n}\in G$ that generate
$G$. 

The \emph{Gasch\"utz rank }of $G$, denoted $\Gas G$, is the minimal
$m\in\mathbb{N}\cup\left\{ \infty\right\} $ such that every open
epimorphism $f:G\to H$ onto a Hausdorff topological group $H$ is
$n$-Gasch\"utz for every $n\geqslant m$.
\end{defn}

Interestingly, we do not know if every $n$-Gasch\"utz map is necessarily
$(n+1)$-Gasch\"utz. This seems to be an interesting question.

Obviously, $\tr G\leqslant\Gas G$ holds for all groups. By definition,
a group is Gasch\"utz precisely when the equality $\tr G=\Gas G$ holds.
The abelian case is settled by the following.
\begin{thm}
\label{thm:into-abel}If $A$ is a connected abelian Lie group with
maximal torus $T$, then 
\[
\Gas A=\begin{cases}
1 & A\text{ is compact,}\\
2\dim A-\dim T & A\text{ is non-compact.}
\end{cases}
\]
\end{thm}

We are thus able to estimate the Gasch\"utz rank of arbitrary connected
Lie groups.
\begin{cor}
\label{intro-cor:gashutz rank}Let $G$ be a connected Lie group.
Then 
\[
\Gas G\leq\max\{\tr G,2\dim(G/\overline{G'})-\dim T\}
\]
where $T$ is the maximal torus of the abelianisation $G/\overline{G'}$.
If $G/\overline{G'}$ is non-compact, then equality holds.
\end{cor}

The value $\tr G$ of a connected Lie group $G$ was computed explicitly
in \cite{AN24}; see Section \ref{sec:AN} for details. Thus, Corollary
\ref{intro-cor:gashutz rank} provides an explicit bound for the Gasch\"utz
rank. It moreover provides a precise value for the Gasch\"utz rank for
most connected Lie groups. Indeed, the only case which remains undetermined
is when $G$ is not compact, $G/\overline{G'}$ is compact, and in
addition $\dim G/\overline{G'}>\tr G$. The nature of this unique
situation is quite subtle. By Corollary \ref{intro-cor:gashutz rank},
$\Gas G$ is at most $\max\{\tr G,2\dim(G/\overline{G'})-\dim T\}$
and at least $\tr G$; both ends of the spectrum are possible. See
\secref{GasRank} for more details.

Corollary \ref{intro-cor:gashutz rank} in particular states that
the Gasch\"utz rank of any connected Lie group is finite. We note that
this is not true for general groups. Indeed, there are connected locally
compact groups (e.g., $G=\mathbb{R}\times(\mathbb{R}/\mathbb{Z})^{\aleph_{0}}$)
with $\tr G<\infty$ but $\Gas G=\infty$ . However, we do not know
the answer to the following question:
\begin{question}
Does there exist a finitely generated discrete group with infinite
Gasch\"utz rank?
\end{question}

The latter is equivalent to asking whether the free group $F_{n}$
on $n$ generators has a finite Gasch\"utz rank for every $n$. 

\subparagraph{The Wiegold Conjecture.}

The Gasch\"utz rank of an epimorphism from $F_{n}$ onto a finite group
is related to the Wiegold conjecture, which we now recall. For a discrete
group $G$, denote by $\mathcal{G}_{n}(G)$ the space of epimorphisms
$F_{n}\to G$, or equivalently the space of generating $n$-tuples
of $G$. We have a natural action of $\mathrm{Aut}(F_{n})$ on $\mathcal{G}_{n}(G)$
by precomposition. The Wiegold conjecture states that the action of
$\mathrm{Aut}(F_{3})$ on $\mathcal{G}_{3}(G)$ is transitive for
every finite simple group $G$. A stronger conjecture, first formulated
as a question by Pak \cite[Question 2.4.6]{Pak01}, states that the
action of $\mathrm{Aut}(F_{n})$ on $\mathcal{G}_{n}(G)$ is transitive
whenever $G$ is a finite group that can be generated by $n-1$ elements.
See \cite{Lub11} for more details.

It is easy to see the following:
\begin{lem}
Let $G$ be a discrete group. The action of $\mathrm{Aut}(F_{n})$
on $\mathcal{G}_{n}(G)$ is transitive if and only if every epimorphism
$F_{n}\to G$ is $n$-Gasch\"utz.
\end{lem}

Therefore, while it is known that $F_{3}$ is not Gasch\"utz, the question
whether every epimorphism from $F_{3}$ onto a finite simple group
is $3$-Gasch\"utz is a major open problem.

In light of this, one readily sees that the analogous question for
profinite groups has a positive solution. That is, the action of $\mathrm{Aut}(\hat{F}_{n})$
on the space of epimorphisms $\hat{F}_{n}\to G$ is always transitive,
since the free profinite group $\hat{F}_{n}$ is Gasch\"utz. This is
a key idea in \cite{Lub01,BCL22} and in \cite[§2.3]{GL09}.

When $G$ is a finite simple group and $n=2$, the action of $\mathrm{Aut}(F_{2})$
on $\mathcal{G}_{2}(G)$ is very often not transitive \cite[Remark 1.4]{Gel24}.
Thus, there exists an epimorphism $F_{2}\to G$ and a generating pair
for $G$ which cannot be lifted to $F_{2}$. However, the abelianisation
map is trivial in this case. This demonstrates that the unique phenomenon
exhibited in \thmref{intro-main} for Lie groups, does not apply for
general groups. 

\subparagraph{The Redundancy Rank.}

Another natural invariant, which was studied before in many contexts,
is the \emph{redundancy rank,} which we denote by $\rho(G)$. It is
the minimal $m\in\mathbb{N}\cup\left\{ \infty\right\} $ such that,
if $g_{1},\dots,g_{n}\in G$ generate $G$ and $n>m$, then one of
the $g_{i}$'s is redundant (i.e., there is some $S\subsetneq\left\{ g_{1},\dots,g_{n}\right\} $
that still generates $G$). 

The redundancy rank has been extensively studied for finite groups,
see \cite{Whi00,LMS00,Luc13.1,Luc13.2,Har23} and \cite{CC02,EG16,HHR07}.
We are not aware of it ever having been given a name. 

In Section \ref{sec:Abelian-groups}, we show the following. 
\begin{thm}
The redundancy rank of a connected abelian Lie group $A$ is equal
to $2\dim A-\dim T$, where $T$ is the maximal torus of $A$. 
\end{thm}

Using this and the results of \cite{AN24}, it is possible to reduce
the computation of the redundancy rank of any connected Lie group
to the computation of the redundancy rank of its maximal semisimple
Lie subgroup. Computing the redundancy rank of these groups seems
to be an intricate problem.
\begin{question}
What is the redundancy rank of a connected simple Lie group?
\end{question}

Curiously, the redundancy rank of $\mathrm{SL}_{2}(\mathbb{R})$ and
$\mathrm{SL}_{2}(\mathbb{C})$ is $\infty$! This was proved in \cite{Min13}
(although it is not explicitly stated there). Indeed, Minsky constructs,
for every $n\geq2$, a generating $n$-tuple in $\mathrm{SL}_{2}(\mathbb{C})$
which is not redundant. In fact, he demonstrates a much stronger property,
namely, that this $n$-tuple remains irredundant even after applying
Nielsen transformations. We refer to \cite{Gel24} for an excellent
survey on actions of $\mathrm{Aut}(F_{n})$ on representation spaces,
covering, among others, the result by Minsky.

It seems reasonable to conjecture that the redundancy rank of any
non-compact semisimple Lie group is infinite. In a future work, we
investigate the redundancy rank of connected compact simple Lie groups,
as well as the \emph{Zariski }redundancy rank of algebraic semisimple
groups, which is analogously defined using the Zariski topology. 

\paragraph{Paper Overview.}

Section \ref{sec:Preliminaries} is a preliminary section in which
we introduce general tools and terminology useful for lifting generators.
In Section \ref{sec:Abelian-groups} we restrict to the case $G$
is abelian. We show that every open epimorphism $f:G\to H$ is $n$-Gasch\"utz
for every $n\geqslant\sr G$, and we show this bound is tight.

In Section \ref{sec:SS} we consider semisimple Lie groups $G$. The
main step there is to show the following.
\begin{prop}
Consider two connected simple Lie groups $G$ and $H$. Then any generating
tuple $x_{1},...,x_{n}\in G$ can be paired with a (generating) tuple
$y_{1},...,y_{n}\in H$ such that $(x_{1},y_{1}),...,(x_{n},y_{n})$
generate the product $G\times H$. 
\end{prop}

The key idea here is the use of suitable word maps to transfer elements
of $G$ into a small `Zassenhaus neighbourhood' of the identity.
We believe the proof of this proposition may be useful for other generation
problems in semisimple Lie groups. 

In \secref{Reductive}, we consider products of semisimple Lie groups
with abelian Lie groups. In Section \ref{sec:AN}, we consider what
we call \emph{Abels--Noskov groups}, a very tame class of groups,
the generating subsets of which were completely classified in \cite{AN24}
(see \defref{abels-noskov} and \thmref{gen-of-AN}). Using their
results, we can reduce the proof of \thmref{intro-main} to Abels--Noskov
groups in Section \ref{sec:GeneralLie}, thus concluding the proof.
Finally, in Section \ref{sec:GasRank} we discuss the Gasch\"utz rank
of connected Lie groups.
\begin{acknowledgement*}
We extend our gratitude to Uri Bader, Tsachik Gelander, Alex Lubotzky
and Gregory Soifer for inspiration and very helpful discussions. We
thank Abels and Noskov for communications regarding this project as
well as their closely related project \cite{AN24} which we strongly
rely on. Sadly, Noskov passed away before we finished the project.
Finally, we would like to express our sincere thanks to the referees
for their attentive reading and insightful comments.

Our research is partially supported by the ERC Advanced Grant 101052954
-- SRS. The second author is supported by NSF postdoctoral fellowship
DMS-2402368. 
\end{acknowledgement*}

\section{Preliminaries\label{sec:Preliminaries}}

In this section we establish our terminology and prove a few basic
lemmas. Throughout this paper, we assume all homomorphisms are continuous;
similarly, if $G$ is a topological group, we say that a subset $S\subseteq G$
\emph{generates $G$ }if it topologically generates $G$: that is,
if $\left\langle S\right\rangle $ (the group abstractly generated
by $S$) is dense in $G$. Recall that $\tr G$, the \emph{(topological)
generating rank }of $G$, is the minimal size of a generating set
of $G$.

As we will see below, it is going to be useful for our purposes to
consider non-Hausdorff groups. Accordingly,\emph{ we do not assume
topological groups are necessarily Hausdorff. }
\begin{notation}
If $f:G\to H$ is a map and $h_{1},\dots,h_{n}\in H$, we say that
$g_{1},\dots,g_{n}\in G$ are \emph{lifts }if $f(g_{i})=h_{i}$ for
$i=1,\dots,n$. 
\end{notation}

\begin{defn}
Let $G$ be a topological group, $f:G\to H$ an epimorphism. We say
that $f$ is \emph{absolutely Gasch\"utz }if, whenever $h_{1},\dots,h_{n}\in H$
generate $H$ (for arbitrary $n\in\mathbb{N}$), and $g_{1},\dots,g_{n}\in G$
are lifts, necessarily $g_{1},\dots,g_{n}$ generate $G$.
\end{defn}

\begin{defn}
Let $G$ be a topological group, $N\trianglelefteqslant G$. We say
that $N$ is \emph{non-generating }if, for every $H\leqslant G$ such
that $HN$ is dense in $G$, necessarily $H$ is dense in $G$.
\end{defn}

The following is immediate:
\begin{lem}
If $N$ is a non-generating normal subgroup of a topological group
$G$, then $f:G\to G/N$ is absolutely Gasch\"utz.
\end{lem}

In particular, if $N\trianglelefteqslant G$ is non-generating, then
$\tr G=\tr{G/N}$. As for lifting generators, we have the following:
\begin{lem}
\label{lem:AbsGas}Let $G$ be a topological group, $f:G\to H$ an
epimorphism. Let $N\trianglelefteqslant G$ be a non-generating normal
subgroup, and consider the induced epimorphism $\tilde{f}:G/N\to H/f(N)$,
which makes the following diagram commute:
\[
\xymatrix{G\ar[r]^{f}\ar[d] & H\ar[d]\\
G/N\ar[r]^{\tilde{f}} & H/f(N)
}
\]
Let $h_{1},\dots,h_{n}\in H$ be generators. Then they admit generating
lifts to $G$ via $f$ if and only if $h_{1}f(N),\dots,h_{n}f(N)$
admit generating lifts to $G/N$ via $\tilde{f}$. 
\end{lem}

Observe that, even if $G$ and $H$ are Hausdorff, $f(N)$ might not
be closed, so $H/f(N)$ might not be Hausdorff.
\begin{proof}
Clearly, if $h_{1},\dots,h_{n}\in H$ admit generating lifts $g_{1},\dots,g_{n}\in G$,
then $g_{1}N,\dots,g_{n}N$ are generating lifts of $h_{1}f(N),\dots,h_{n}f(N)$. 

Now, in order to show the opposite direction, let $h_{1},\dots,h_{n}\in H$
be generators of $G$ such that $h_{1}f(N),\dots,h_{n}f(N)$ admit
generating lifts $g_{1}N,\dots,g_{n}N\in G/N$. Since $G\to G/N$
is absolutely Gasch\"utz, we get that $g_{1},\dots,g_{n}$ generate
$G$. It might seem that we are done, but not quite: $g_{1},\dots,g_{n}$
might not be lifts of $h_{1},\dots,h_{n}$. All we know is that $f(g_{i})f(N)=h_{i}f(N)$
for $i=1,\dots,n$. But this means that there are $\varepsilon_{1},\dots,\varepsilon_{n}\in N$
such that $f(g_{i}\varepsilon_{i})=h_{i}$ for $i=1,\dots,n$. Since
$G\to G/N$ is absolutely Gasch\"utz, the elements $g_{1}\varepsilon_{1},\dots,g_{n}\varepsilon_{n}$
still generate $G$, so we are done. 
\end{proof}
It will be useful to consider non-generating normal subgroups $N$
whose image in some quotients might not be closed; therefore, it is
useful to consider quotients by non-closed normal subgroups as well,
if we want to be able to quotient out any non-generating normal subgroup
we like. Luckily, we don't need to investigate quotients by \emph{any
}non-closed normal subgroup, such as $\mathbb{Q}$ in $\mathbb{R}$.
Consider the diagram in the last lemma: while $\ker\tilde{f}$ might
not be closed, it is still the image of $\ker f$ under the quotient
$G\to G/N$, and $\ker f$ \emph{is }closed. This gives us some information
on $\ker f$, and hence $\ker\tilde{f}$. For instance, they are finitely
generated: 
\begin{lem}
\label{lem:NormalisFG}A closed normal subgroup of a connected Lie
group is (topologically) finitely generated. 
\end{lem}

\begin{proof}
Let $N$ be a closed, normal subgroup of a connected Lie group $G$.
Its connected component $\conn N$ is a connected Lie group, hence
finitely generated. Therefore, it is enough to show $N/\conn N$ is
finitely generated. Since $\conn N$ is characteristic in $N$, it
is normal in $G$, and $N/\conn N$ is a normal subgroup of $G/\conn N$.
Thus, it is enough to show that discrete, normal subgroups of connected
Lie groups are finitely generated.

Let $\Gamma\trianglelefteqslant G$ be a discrete, normal subgroup
of a connected Lie group $G$. Then $\Gamma$ is central in $G$.
Consider the universal cover $\tilde{G}$ of $G$, and denote by $\pi:\tilde{G}\to G$
the covering map. Since $\pi$ is a local homeomorphism, we get that
$\pi^{-1}(\Gamma)$ is discrete. It is also normal in $\tilde{G}$,
hence central. Now, consider $\tilde{G}/\pi^{-1}(\Gamma)$; since
$\tilde{G}$ is simply-connected, $\pi^{-1}(\Gamma)$ is the fundamental
group of $\tilde{G}/\pi^{-1}(\Gamma)$. It is well-known that every
connected Lie group is homotopy equivalent to a maximal compact subgroup,
and that the fundamental group of a compact manifold is finitely generated.
Since $\Gamma$ is a quotient of $\pi^{-1}(\Gamma)$, we get that
$\Gamma$ is finitely generated.
\end{proof}

Observe that, if $\Delta\trianglelefteqslant G$ is a normal subgroup
of a topological group $G$, then $g_{1},\dots,g_{n}\in G$ descend
to generators of $G/\Delta$ if and only if they descend to generators
of $G/\overline{\Delta}$, if and only if $\left\langle g_{1},\dots,g_{n}\right\rangle \Delta$
is dense. 

It will sometimes be more convenient to focus on subgroups rather
than maps. So, instead of considering an epimorphism $f:G\to H$ and
looking for lifts of some generators $h_{1},\dots,h_{n}\in H$, we
may begin with elements $g_{1},\dots,g_{n}\in G$ for which $\left\langle g_{1},\dots,g_{n}\right\rangle \ker f$
is dense in $G$ (i.e., elements descending to generators of $H$)
and look for $k_{1},\dots,k_{n}\in\ker f$ such that $g_{1}k_{1},\dots,g_{n}k_{n}$
generate $G$.

\section{Abelian Groups\label{sec:Abelian-groups}}

If $A$ is a connected abelian Lie group, then its maximal compact
subgroup $T$ is always isomorphic to a torus. Putting $m=\dim T$
and $n=\dim A-\dim T$, we always have $A\cong\mathbb{R}^{n}\times\mathbb{T}^{m}$. 
\begin{notation}
For the purposes of this section, the \emph{free rank }of a connected
abelian Lie group $A$ is $\dim A/T$, where $T$ is its maximal compact
subgroup.
\end{notation}

The purpose of this section is to show that, if $f:\mathbb{R}^{n}\times\mathbb{T}^{m}\to H$
is an open epimorphism and $h_{1},\dots,h_{k}\in H$ are generators
of $H$ with $k\geqslant2n+m$, there are lifts $g_{1},\dots,g_{k}\in\mathbb{R}^{n}\times\mathbb{T}^{m}$
that generate $\mathbb{R}^{n}\times\mathbb{T}^{m}$. 

If $n=0$ and $H$ is Hausdorff, the requirement $k\geqslant2n+m$
can be removed (e.g., by \cite{CG18}); however, we will see that
if either $n>0$ or $H$ is allowed not to be Hausdorff, there is
an epimorphism $f:\mathbb{R}^{n}\times\mathbb{T}^{m}\to H$ and generators
$h_{1},\dots,h_{2n+m-1}\in H$ of $H$ that do not admit generating
lifts (see \propref{lowerbnd} and \lemref{GasStarofaTorus}).

We begin with two standard lemmas. Recall that $v_{1},\dots,v_{n}\in\mathbb{R}^{n}$
are a basis if and only if $\left\langle v_{1},\dots,v_{n}\right\rangle \leqslant\mathbb{R}^{n}$
is a lattice.
\begin{lem}
\label{lem:latticeprojection}Let $G$ be a connected abelian Lie
group with maximal compact subgroup $K$, and let $\pi:G\to G/K$
be the projection. Set $n=\dim G/K$ and let $g_{1},\dots,g_{n}\in G$.
Then $g_{1},\dots,g_{n}$ generate a lattice in $G$ if and only if
$\pi(g_{1}),\dots,\pi(g_{n})$ generate a lattice in $G/K$.
\end{lem}

\begin{proof}
If $\left\langle g_{1},\dots,g_{n}\right\rangle $ is a lattice then
is it is particular closed, so $\left\langle g_{1},\dots,g_{n}\right\rangle K$
is closed, so $\pi(\left\langle g_{1},\dots,g_{n}\right\rangle )=\left\langle \pi(g_{1}),\dots,\pi(g_{n})\right\rangle $
is closed. Since it is countable and $G/K$ is locally compact, it
is discrete. Since $\left\langle g_{1},\dots,g_{n}\right\rangle $
is cocompact, $\left\langle \pi(g_{1}),\dots,\pi(g_{n})\right\rangle $
is cocompact as well, hence a lattice.

Now, assume $\left\langle \pi(g_{1}),\dots,\pi(g_{n})\right\rangle $
is a lattice in $G/K$. Recall that $G/K\cong\mathbb{R}^{n}$, so
$(\pi(g_{1}),\dots,\pi(g_{n}))$ is a basis. Denote $E=\{g_{1},\dots,g_{n}\}$;
we claim $\left\langle E\right\rangle $ is discrete. First, observe
that $\left\langle E\right\rangle $ intersects $K$ trivially (if
$x\in\left\langle E\right\rangle \cap K$, then $x=\sum_{i=1}^{n}\alpha_{i}g_{i}$
for some $\alpha_{1},\dots,\alpha_{n}\in\mathbb{Z}$ and also $\pi(x)=\sum_{i=1}^{n}\alpha_{i}\pi(g_{i})=0$;
since $\pi(g_{1}),\dots,\pi(g_{n})$ are linearly independent, this
means $\alpha_{i}=0$ for every $i$, which means $x=0$). If $\left\langle E\right\rangle $
is not discrete, then there is a sequence of nontrivial elements in
$\left\langle E\right\rangle $ approaching zero; since $\pi(\left\langle E\right\rangle )$
is discrete, almost all of them must be mapped to zero under $\pi$.
This is a contradiction, since we just showed $\pi$ is injective
on $\left\langle E\right\rangle $. Since $K$ is compact and $\left\langle \pi(E)\right\rangle $
is cocompact, we get that $\left\langle E\right\rangle $ is discrete
and cocompact, as needed.
\end{proof}
\begin{lem}
\label{lem:generation-of-RnTm}A subset $S\subseteq\mathbb{R}^{n}\times\mathbb{T}^{m}$
generates $\mathbb{R}^{n}\times\mathbb{T}^{m}$ if and only if there
is some $E\subseteq S$ projecting to a basis of $\mathbb{R}^{n}$
such that $S\backslash E$ projects to a generating set of $(\mathbb{R}^{n}\times\mathbb{T}^{m})/\langle E\rangle$.
Moreover, if $E$ has exactly $n$ elements, then $(\mathbb{R}^{n}\times\mathbb{T}^{m})/\langle E\rangle$
is isomorphic to $\mathbb{T}^{n+m}$.
\end{lem}

\begin{proof}
Denote by $p_{1}:\mathbb{R}^{n}\times\mathbb{T}^{m}\to\mathbb{R}^{n}$
the projection and set $G=\mathbb{R}^{n}\times\mathbb{T}^{m}$, $K=\mathbb{T}^{m}$.
If $S\subseteq G$ generates $G$, then $p_{1}(S)$ generates $\mathbb{R}^{n}$,
so clearly there is some $E\subseteq S$ such that $p_{1}(E)$ is
a basis of $\mathbb{R}^{n}$, and obviously $S\backslash E$ projects
to a generating set of $G/\langle E\rangle$ (since $S$ does, and
$E$ projects to the trivial element there). Conversely, assume there
is $E\subseteq S$ such that $p_{1}(E)$ is a basis of $\mathbb{R}^{n}$
and $S\backslash E$ projects to a generating set of $G/\langle E\rangle$.
Denote $\tau:G\to G/\langle E\rangle$. Since $\tau(\left\langle S\right\rangle )$
is dense, we get that $\ker\tau\cdot\left\langle S\right\rangle $
is dense in $G$; but $\ker\tau\subseteq\left\langle S\right\rangle $,
so we are done.

The `moreover' part follows from the previous lemma: if $p_{1}(E)$
is a basis of $\mathbb{R}^{n}$, then $\left\langle E\right\rangle $
is a lattice, so that $G/\left\langle E\right\rangle $ is a torus
of dimension $n+m$.
\end{proof}

\subsection{The Redundancy Rank}
\begin{defn}
A generating subset $S$ of a topological group $G$ is \emph{irredundant
}if, for every $s\in S$, the set $S\backslash\left\{ s\right\} $
no longer generates $G$. The \emph{redundancy rank }of a topological
group $G$, denoted $\sr G$, is the maximal size of an irredundant
finite generating subset. In other words, it is the minimal $m\in\mathbb{N}\cup\left\{ \infty\right\} $
such that every generating subset $S$ with $\left|S\right|>m$ is
redundant. 
\end{defn}

In this subsection we calculate the redundancy rank of connected abelian
Lie groups.
\begin{prop}
\label{prop:sr-of-RnTm}The redundancy rank of $\mathbb{R}^{n}\times\mathbb{T}^{m}$
is $2n+m$ (for every $n,m\geqslant0$).
\end{prop}

\begin{proof}
We need to prove two things: one, there is an irredundant generating
set of $\mathbb{R}^{n}\times\mathbb{T}^{m}$ of size (at least) $2n+m$,
and two, every finite generating set of $\mathbb{R}^{n}\times\mathbb{T}^{m}$
of size strictly greater than $2n+m$ has a redundant element.

First, by Kronecker's theorem it's easy to see that for any irrational
$\alpha\in\mathbb{R}$ the sets $\{\alpha+\mathbb{Z}\}$ and $\{\alpha,1\}$
are irredundant generating sets of $\mathbb{T}$ and $\mathbb{R}$
respectively. Now, observe that if $S$ is an irredundant generating
set of a group $G$ and $T$ is an irredundant generating set of a
group $H$ then $S\times\{1_{H}\}\cup\{1_{G}\}\times T$ is an irredundant
generating set of $G\times H$. Thus we can construct an irredundant
generating set of $\mathbb{R}^{n}\times\mathbb{T}^{m}$ of size $2n+m$,
as needed.

Now, let $S$ be a generating set of $\mathbb{R}^{n}\times\mathbb{T}^{m}$,
$|S|>2n+m$. By Lemma \ref{lem:generation-of-RnTm}, there is $E\subseteq S$
of size $n$ such that $S\backslash E$ projects to a generating set
of $(\mathbb{R}^{n}\times\mathbb{T}^{m})/\langle E\rangle\cong\mathbb{T}^{n+m}$.
So we just need to show that every generating set of a torus $G$
of size strictly greater than $\dim G$ has a redundant element. We
prove this by induction on $\dim G$. Assume the claim is true for
tori of dimension strictly smaller than $\dim G$, and let us prove
it for $G$. Suppose $S\subseteq G$ generates $G$ and $|S|>\dim G$.
Pick some $s_{0}\in S$, and let $n$ be the index of $\conn{\overline{\langle S\backslash\{s_{0}\}\rangle}}$
inside $\overline{\langle S\backslash\{s_{0}\}\rangle}$. Then $T=\left\langle ns\right\rangle _{s\in S\backslash\{s_{0}\}}$
is of finite index in $\overline{\langle S\backslash\{s_{0}\}\rangle}$
and is contained in $\conn{\overline{\langle S\backslash\{s_{0}\}\rangle}}$,
hence it is equal to $\conn{\overline{\langle S\backslash\{s_{0}\}\rangle}}$.
If $T=G$, then $s_{0}$ is redundant and we are done; otherwise,
$\dim T<\dim G$ and $|\{ns\}_{s\in S\backslash\{s_{0}\}}|>\dim T,$
so, by the induction hypothesis, there is a redundant element.
\end{proof}

\subsection{Lifting Generators}

First, let us recall the following fact: 
\begin{lem}
\label{lem:quotients-by-connected-split-(abelian)}Let $G$ be a connected
abelian Lie group, $A\leqslant G$ a closed connected subgroup. Then
there is a closed connected subgroup $B$ such that the addition map
$A\times B\to G$ is an isomorphism.
\end{lem}

\begin{lem}
\label{lem:correction}Let $G$ be an $n$-dimensional torus. Let
$x_{1},\dots,x_{n}\in G$ be some elements, and let $\Delta\leqslant G$
be a (not-necessarily-closed) subgroup which is topologically finitely
generated. Assume $\overline{\left\langle x_{i}\right\rangle _{i=1}^{n}}+\overline{\Delta}=G$.
Then there are $\delta_{1},\dots,\delta_{n}\in\Delta$ such that $x_{1}+\delta_{1},\dots,x_{n}+\delta_{n}$
are topological generators of $G$.
\end{lem}

\begin{proof}
Set $A=\overline{\left\langle x_{i}\right\rangle _{i=1}^{n}}$, and
set $k=[A:\conn A]$. The group $\overline{\left\langle kx_{i}\right\rangle _{i=1}^{n}}$
is certainly contained in $\conn A$, and is of finite index in it,
and hence equals $\conn A$. Set $m=\dim A$; the redundancy rank
of $A\cong\mathbb{T}^{m}$ is $m$, so up to rearrangement we may
assume $\conn A=\overline{\left\langle kx_{i}\right\rangle _{i=1}^{m}}$.

Observe that $\conn{\overline{\Delta}}\cong\mathbb{T}^{\dim\overline{\Delta}}$.
Let $\delta_{1},\dots,\delta_{N}$ be topological generators of $\Delta$.
Then they also topologically generate $\overline{\Delta}$. Therefore,
setting $\ell=[\overline{\Delta}:\conn{\overline{\Delta}}]$, we get
that $\ell\delta_{1},\dots,\ell\delta_{N}$ topologically generate
$\conn{\overline{\Delta}}$. Thus, since the redundancy rank of $\conn{\overline{\Delta}}$
is $\dim\overline{\Delta}$ and $\dim\overline{\Delta}\leqslant n$,
there are some $n$ elements among $\ell\delta_{1},\dots,\ell\delta_{N}$
which already topologically generate $\conn{\overline{\Delta}}$.
Denote these elements (which are contained in $\Delta$) by $y_{1},\dots,y_{n}$,
so that $\overline{\left\langle y_{i}\right\rangle _{i=1}^{n}}=\conn{\overline{\Delta}}$.

Next, observe that, since $\conn{\overline{\Delta}}$ is of finite
index in $\overline{\Delta}$ and $\conn A$ is of finite index in
$A$, the group $\conn A+\conn{\overline{\Delta}}$ is of finite index
in $A+\overline{\Delta}=G$. Since $G$ is connected, this means $G=\conn A+\conn{\overline{\Delta}}$,
so $G$ is topologically generated by $\left\{ kx_{i}\right\} _{i=1}^{n}\cup\left\{ y_{j}\right\} _{j=1}^{n}$,
and by the exact same argument one sees that $G$ is also topologically
generated by $\left\{ kx_{i}\right\} _{i=1}^{n}\cup\left\{ ky_{j}\right\} _{j=1}^{n}$
(since once again the group topologically generated by this set is
of finite index in $G$). Let $B$ be a closed connected subgroup
of $G$ such that $G=\conn A+B$ and $\conn A\cap B=\left\{ 0\right\} $,
and consider the well-defined projection $q:G\to B$. Since $\left\{ kx_{i}\right\} _{i=1}^{n}\cup\left\{ ky_{j}\right\} _{j=1}^{n}$
topologically generates $G$, its projection in $B$ topologically
generates $B$; since $q(kx_{i})=0$ for each $i$, we see that $\left\{ q(ky_{j})\right\} _{j=1}^{n}$
topologically generates $B$. Since the redundancy rank of $B\cong\mathbb{T}^{n-m}$
is $n-m$, we may assume (up to rearranging the $y_{j}$'s) that
$B$ is already generated by $\left\{ q(ky_{j})\right\} _{j=m+1}^{n}$.
Since $\overline{\left\langle kx_{i}\right\rangle _{i=1}^{m}}=\conn A$,
it easily follows $\left\{ kx_{i}\right\} _{i=1}^{m}\cup\left\{ ky_{j}\right\} _{j=m+1}^{n}$
topologically generates $G$.

We are almost done. We now claim $\left\{ x_{i}\right\} _{i=1}^{m}\cup\left\{ x_{j}+y_{j}\right\} _{j=m+1}^{n}$
topologically generates $G$. This is immediate: the group topologically
generated by this set obviously contains $kx_{i}$ for $i=1,\dots,m$,
and hence contains $\conn A$, and hence contains $kx_{j}$ for $j=m+1,\dots,n$
as well. Therefore, since it contains $x_{j}+y_{j}$ for $j=m+1,\dots,n$,
it must also contain $ky_{j}$ for $j=m+1,\dots,n$. But we saw $\left\{ kx_{i}\right\} _{i=1}^{m}\cup\left\{ ky_{j}\right\} _{j=m+1}^{n}$
topologically generates $G$, so we are done.
\end{proof}
\begin{lem}
\label{lem:lifting-basis}Let $G$ be a connected abelian Lie group,
and let $\Delta$ be a (not-necessarily-closed) subgroup. Let $n$
and $\ell$ be the free ranks of $G$ and $G/\overline{\Delta}$ respectively.
If $x_{1},\dots,x_{\ell}\in G$ are such that their projections in
$G/\overline{\Delta}$ generate a lattice in $G/\overline{\Delta}$,
and $x_{\ell+1},\dots,x_{n}\in G$ are any elements, then there are
$\delta_{1},\dots,\delta_{n}\in\Delta$ such that $x_{1}+\delta_{1},\dots,x_{n}+\delta_{n}$
generate a lattice in $G$. 
\end{lem}

\begin{proof}
Set $H=G/\overline{\Delta}$. Let $K_{G},K_{H}$ be the maximal tori
of $G$ and $H$ respectively. Then $G/K_{G}\cong\mathbb{R}^{n}$
and $H/K_{H}\cong\mathbb{R}^{\ell}$. Since the image of $K_{G}$
in $H$ is contained in $K_{H}$, we have a well-defined epimorphism
$\varphi:G/K_{G}\to H/K_{H}$, which is a linear map (as all continuous
homomorphisms between vector groups over $\mathbb{R}$). Denoting
by $\pi:G\to G/K_{G}$ and $p:H\to H/K_{H}$ the projections, we get
the following commutative diagram:
\[
\xymatrix{G\ar[r]\ar[d]_{\pi} & G/\Delta\ar[r] & H\ar[d]^{p}\\
G/K_{G}\ar[rr]_{\varphi} &  & H/K_{H}
}
\]
Denote by $\bar{x}_{1},\dots,\bar{x}_{n}$ the projections of $x_{1},\dots,x_{n}$
in $H$. Since $\bar{x}_{1},\dots,\bar{x}_{\ell}$ generate a lattice
in $H$, their projections in $H/K_{H}$ generate a lattice in $H/K_{H}$.
This means that $\left\{ p(\bar{x}_{i})\right\} _{i=1}^{\ell}$ is
a basis of $H/K_{H}$ when viewed as a vector space. Since $\varphi$
is linear, we obtain that $\pi(x_{1}),\dots,\pi(x_{\ell})\in G/K_{G}$
are linearly independent when we think of $G/K_{G}$ as a vector space
(since they are lifts of $p(x_{1}),\dots,p(x_{\ell})$). 

Next, we claim $\mathrm{Span}(\pi(\Delta))=\ker\varphi$. Clearly
$\mathrm{Span}(\pi(\Delta))\subseteq\ker\varphi$. Now, suppose $\pi(g)\in\ker\varphi$;
this means that the image of $g$ in $H$ is contained in $K_{H}$,
which implies that for every identity neighbourhood $W$ of $H$ there
is some nontrivial power of the image of $g$ which lies in it. Therefore,
for every neighbourhood $W$ of $\overline{\Delta}$, there is some
$n\geqslant1$ such that $g^{n}\in W$. Clearly the same holds for
neighbourhoods of $\Delta$. This means that for every neighbourhood
$W$ of $\pi(\Delta)$ there is some $n\geqslant1$ such that $n\pi(g)\in W$.
Definitely the same holds for $\mathrm{Span}(\pi(\Delta))$, which
is only bigger. But this obviously means that $ng\in\mathrm{Span}(\pi(\Delta))$
for some $n\geqslant1$, so $\pi(g)\in\mathrm{Span}(\pi(\Delta))$,
as needed. 

Thus, we can extend the linearly independent set $\left\{ \pi(x_{i})\right\} _{i=1}^{\ell}$
to a basis of $G/K_{G}$ using elements from $\pi(\Delta)$. In other
words, there are $\varepsilon_{\ell+1},\dots,\varepsilon_{n}\in\Delta$
such that $\left\{ \pi(x_{i})\right\} _{i=1}^{\ell}\cup\left\{ \pi(\varepsilon_{j})\right\} _{j=\ell+1}^{n}$
is a vector space basis of $G/K_{G}$. Therefore, by the continuity
of the determinant, there is some large $N\in\mathbb{N}$ such that
$\left\{ \pi(x_{i})\right\} _{i=1}^{\ell}\cup\left\{ \pi(x_{j}+N\varepsilon_{j})\right\} _{j=\ell+1}^{n}$
is also a basis of $G/K_{G}$. 

(To be more precise, consider the vector space decomposition $G/K_{G}=\mathrm{Span}\left\{ \pi(x_{i})\right\} _{i=1}^{\ell}\oplus\mathrm{Span}\left\{ \pi(\varepsilon_{j})\right\} _{j=\ell+1}^{n}$,
which gives us a linear projection $G/K_{G}\to\mathrm{Span}\left\{ \pi(\varepsilon_{j})\right\} _{j=\ell+1}^{n}$.
The projection of $\left\{ \pi(\varepsilon_{j})\right\} _{j=\ell+1}^{n}$
in $\mathrm{Span}\left\{ \pi(\varepsilon_{j})\right\} _{j=\ell+1}^{n}$
is a basis, so, if $N\in\mathbb{N}$ is large enough, the projection
of $\left\{ \frac{1}{N}\pi(x_{j})+\pi(\varepsilon_{j})\right\} _{j=\ell+1}^{n}$
in $\mathrm{Span}\left\{ \pi(\varepsilon_{j})\right\} _{j=\ell+1}^{n}$
is still a basis of $\mathrm{Span}\left\{ \pi(\varepsilon_{j})\right\} _{j=\ell+1}^{n}$.
Therefore $\left\{ \pi(x_{i})\right\} _{i=1}^{\ell}\cup\left\{ \frac{1}{N}\pi(x_{j})+\pi(\varepsilon_{j})\right\} _{j=\ell+1}^{n}$
is a basis of $G/K_{G}$, and hence $\left\{ \pi(x_{i})\right\} _{i=1}^{\ell}\cup\left\{ \pi(x_{j}+N\varepsilon_{j})\right\} _{j=\ell+1}^{n}$
is a basis as well.)

Denote $E=\{x_{1},\dots,x_{\ell},x_{\ell+1}+N\varepsilon_{\ell+1},\dots,x_{n}+N\varepsilon_{n}\}$,
so that $\left\langle \pi(E)\right\rangle $ is a lattice of $G/K_{G}$,
and hence $\left\langle E\right\rangle $ is a lattice of $G$ (by
\lemref{latticeprojection}), and we are done (by setting $\delta_{i}=0$
for $i=1,\dots,\ell$ and $\delta_{i}=N\varepsilon_{i}$ for $i=\ell+1,\dots,n$).
\end{proof}
\begin{prop}
\label{prop:AbelianCase}Let $G$ be a connected abelian Lie group
with maximal torus $K$. If $\Delta$ is a (not-necessarily-closed)
topologically-finitely-generated subgroup, $n\geqslant2\dim G-\dim K$
and $g_{1},\dots,g_{n}\in G$ are such that $\left\langle g_{1},\dots,g_{n}\right\rangle \Delta$
is dense in $G$, then there are $\delta_{1},\dots,\delta_{n}\in\Delta$
such that $g_{1}+\delta_{1},\dots,g_{n}+\delta_{n}$ topologically
generate $G$.
\end{prop}

\begin{proof}
Set $H=G/\overline{\Delta}$, and denote by $h_{1},\dots,h_{n}$ the
images of $g_{1},\dots,g_{n}$, which are topological generators of
$H$. By assumption, $n\geqslant\sr G$ (the redundancy rank of $G$). 

Let $k$ and $\ell$ be the free ranks of $G$ and $H$ respectively,
and let $K_{G},K_{H}$ be the tori of $G$ and $H$ respectively (so
$G/K_{G}\cong\mathbb{R}^{k}$ and $H/K_{H}\cong\mathbb{R}^{\ell}$),
and observe that $\ell\leqslant k$. Set $d=\dim G$, $m=\dim H$,
so that $\sr G=d+k$ and $\sr H=m+\ell$.

Since $h_{1},\dots,h_{n}$ topologically generate $H$, their projections
topologically generate $H/K_{H}$, and in particular span it as a
vector space. Therefore, we may assume (up to rearrangement) that
$h_{1},\dots,h_{\ell}$ project to a basis of $H/K_{H}$; this implies
that their projections generate a lattice in $H/K_{H}$, and hence
they generate a lattice in $H$. Now, since $h_{1},\dots,h_{n}$ topologically
generate $H$ and $h_{1},\dots,h_{\ell}$ project to the trivial element
in $H/\left\langle h_{i}\right\rangle _{i=1}^{\ell}$, we get that
$h_{\ell+1},\dots,h_{n}$ project to topological generators of $H/\left\langle h_{i}\right\rangle _{i=1}^{\ell}\cong\mathbb{T}^{m}$.
Since the redundancy rank of $H/\left\langle h_{i}\right\rangle _{i=1}^{\ell}$
is $m$, we can rearrange the $h_{i}$'s so that already $h_{n-m+1},\dots,h_{n}$
project to topological generators of $H/\left\langle h_{i}\right\rangle _{i=1}^{\ell}$.
In particular, the $\ell+m$ elements $h_{1},\dots,h_{\ell},h_{n-m+1},\dots,h_{n}$
already topologically generate $H$. Observe that $n\geqslant d+k\geqslant m+\ell$,
and hence there is no overlapping between the two sequences $1,\dots,\ell$
and $n-m+1,\dots,n$. In fact, since $n\geqslant d+k\geqslant m+k$,
we even get that there is no overlapping between $1,\dots,k$ and
$n-m+1,\dots,n$ (while $k\geqslant\ell$); we will use this below.

By \lemref{lifting-basis}, there are $\delta_{1},\dots,\delta_{k}\in\Delta$
such that $g_{1}+\delta_{1},\dots,g_{k}+\delta_{k}\in G$ generate
a lattice in $G$. Now, clearly $g_{1},\dots,g_{\ell},g_{n-m+1},\dots,g_{n}$
together with $\Delta$ generate $G$ (since $h_{1},\dots,h_{\ell},h_{n-m+1},\dots,h_{n}$
generate $H$ and $\Delta$ is dense in $\ker(G\to H)$), so (since
$g_{1},\dots,g_{k}$ project to nothing), $g_{n-m+1},\dots,g_{n}$
together with $\Delta$ project to generators of the torus $G/\left\langle g_{1},\dots,g_{k}\right\rangle $.
Therefore \lemref{correction} does the trick.
\end{proof}

\subsection{Tightness}

We now show that the bound given by \propref{AbelianCase} is tight
in the case G is non-compact.
\begin{prop}
\label{prop:lowerbnd}For every $n,m\geqslant0$ there exists an open
epimorphism $f:\mathbb{R}^{n+1}\times\mathbb{T}^{m}\to H$ onto a
(Hausdorff) Lie group $H$ and a generating system $\underline{h}\in H^{2n+m+1}$
that doesn't admit a generating lift.
\end{prop}

We will need the following general version of Kronecker's theorem:
\begin{lem}[Kronecker's theorem]
\label{lem:Kronecker}For any index set $I$, the elements $(\alpha_{1}^{i}+\mathbb{Z})_{i\in I},\dots,(\alpha_{m}^{i}+\mathbb{Z})_{i\in I}$
in $\mathbb{T}^{I}$ generate $\mathbb{T}^{I}$ if and only if the
following implication holds: if $\lambda_{1},\dots,\lambda_{n}\in\mathbb{Q}$
and $i_{1},\dots,i_{n}\in I$ are such that
\begin{gather*}
\lambda_{1}\alpha_{1}^{i_{1}}+\cdots+\lambda_{n}\alpha_{1}^{i_{n}}\in\mathbb{Q}\\
\vdots\\
\lambda_{1}\alpha_{m}^{i_{1}}+\cdots+\lambda_{n}\alpha_{m}^{i_{n}}\in\mathbb{Q}
\end{gather*}
then $\lambda_{j}=0$ for every $j=1,\dots,n$.
\end{lem}

\begin{proof}[Proof of \lemref{Kronecker}]
This is immediate from Pontryagin duality.
\end{proof}
\begin{proof}[Proof of \propref{lowerbnd}]
Let us write $G=(\mathbb{R}^{n}\times\mathbb{T}^{m})\times\mathbb{R}$
and set $H=(\mathbb{R}^{n}\times\mathbb{T}^{m})\times\mathbb{T}$,
and consider the natural quotient $f:G\to H$. Set 
\begin{align*}
\underline{h} & =\begin{pmatrix}| &  & | & | &  & | & |\\
e_{1} & \cdots & e_{n} & y_{1} & \cdots & y_{n+m} & x\\
| &  & | & | &  & | & |
\end{pmatrix}\\
 & =\begin{pmatrix}\begin{matrix}1 &  & 0\\
 & \ddots\\
0 &  & 1
\end{matrix} & \vline & \begin{matrix}\sqrt{2} &  & 0\\
 & \ddots\\
0 &  & \sqrt{2}
\end{matrix} & \vline & 0 & \vline & \begin{matrix}0\\
\vdots\\
0
\end{matrix}\\
\hline 0 & \vline & 0 & \vline & \begin{matrix}\sqrt{2}+\mathbb{Z} &  & 0\\
 & \ddots\\
0 &  & \sqrt{2}+\mathbb{Z}
\end{matrix} & \vline & \begin{matrix}0\\
\vdots\\
0
\end{matrix}\\
\hline 0 & \vline & 0 & \vline & 0 & \vline & \sqrt{2}+\mathbb{Z}
\end{pmatrix}
\end{align*}
which is in $H^{2n+m+1}=H^{\sr G-1}$, and is a generating system
of $H$, since $y_{1},\dots,y_{n+m},x$ project to generators of $H/\langle e_{i}\rangle_{i=1}^{n}$.
Assume by contradiction there is a generating lift $\underline{g}=(\hat{e}_{1},\dots,\hat{e}_{n},\hat{y}_{1},\dots,\hat{y}_{n+m},\hat{x})\in G^{2n+m+1}$;
it's of the form
\[
\begin{pmatrix}\begin{matrix}1 &  & 0\\
 & \ddots\\
0 &  & 1
\end{matrix} & \vline & \begin{matrix}\sqrt{2} &  & 0\\
 & \ddots\\
0 &  & \sqrt{2}
\end{matrix} & \vline & 0 & \vline & \begin{matrix}0\\
\vdots\\
0
\end{matrix}\\
\hline 0 & \vline & 0 & \vline & \begin{matrix}\sqrt{2}+\mathbb{Z} &  & 0\\
 & \ddots\\
0 &  & \sqrt{2}+\mathbb{Z}
\end{matrix} & \vline & \begin{matrix}0\\
\vdots\\
0
\end{matrix}\\
\hline \begin{matrix}k_{1} & \cdots & k_{n}\end{matrix} & \vline & \begin{matrix}\ell_{1} & \cdots & \ell_{n}\end{matrix} & \vline & \begin{matrix}\ell_{n+1} & \cdots & \ell_{n+m}\end{matrix} & \vline & \sqrt{2}+k_{0}
\end{pmatrix}
\]
for some $k_{0},\dots,k_{n},\ell_{1},\dots,\ell_{n+m}\in\mathbb{Z}$.
Set $S=\{\hat{e}_{i}\}_{i=1}^{n}\cup\{\hat{x}\}$; since $\underline{g}$
generates $G$, we know in particular that $\{\hat{y}_{i}\}_{i=1}^{n+m}$
projects to a generating set of $G/\langle S\rangle$. We show this
can't be true.

Consider the universal covering $\tilde{G}=\mathbb{R}^{n+m+1}$, and
lift $\underline{g}$ to $\underline{\tilde{g}}=(\tilde{e}_{1},\dots,\tilde{e}_{n},\tilde{y}_{1},\dots,\tilde{y}_{n+m},\tilde{x})\in\tilde{G}^{2n+m+1}$
thusly:
\[
\underline{\tilde{g}}=\begin{pmatrix}\begin{matrix}1 &  & 0\\
 & \ddots\\
0 &  & 1
\end{matrix} & \vline & \begin{matrix}\sqrt{2} &  & 0\\
 & \ddots\\
0 &  & \sqrt{2}
\end{matrix} & \vline & 0 & \vline & \begin{matrix}0\\
\vdots\\
0
\end{matrix}\\
\hline 0 & \vline & 0 & \vline & \begin{matrix}\sqrt{2} &  & 0\\
 & \ddots\\
0 &  & \sqrt{2}
\end{matrix} & \vline & \begin{matrix}0\\
\vdots\\
0
\end{matrix}\\
\hline \begin{matrix}k_{1} & \cdots & k_{n}\end{matrix} & \vline & \begin{matrix}\ell_{1} & \cdots & \ell_{n}\end{matrix} & \vline & \begin{matrix}\ell_{n+1} & \cdots & \ell_{n+m}\end{matrix} & \vline & \sqrt{2}+k_{0}
\end{pmatrix}.
\]
Set $\varGamma=\langle\tilde{e}_{i}\rangle_{i=1}^{n}+\langle e_{i}\rangle_{i=n+1}^{n+m}+\langle\tilde{x}\rangle,$
so that $\tilde{G}/\varGamma=G/\langle S\rangle$ and $\{\tilde{y}_{i}\}_{i=1}^{n+m}$
projects to a generating set of $\tilde{G}/\varGamma$. From here
we will arrive at a contradiction. Set
\begin{align*}
\varPsi & =\begin{pmatrix}| &  & | & | &  & | & |\\
\tilde{e}_{1} & \cdots & \tilde{e}_{n} & e_{n+1} & \cdots & e_{n+m} & \tilde{x}\\
| &  & | & | &  & | & |
\end{pmatrix}\\
 & =\begin{pmatrix}I_{n} & \vline & 0 & \vline & \begin{matrix}0\\
\vdots\\
0
\end{matrix}\\
\hline 0 & \vline & I_{m} & \vline & \begin{matrix}0\\
\vdots\\
0
\end{matrix}\\
\hline \begin{matrix}k_{1} & \cdots & k_{n}\end{matrix} & \vline & \begin{matrix}0 & \cdots & 0\end{matrix} & \vline & \sqrt{2}+k_{0}
\end{pmatrix}
\end{align*}
so that $\varPsi^{-1}$ sends $\varGamma$ to $\mathbb{Z}^{n+m+1}$,
which means $\{\varPsi^{-1}\tilde{y}_{i}\}_{i=1}^{n+m}$ generates
$\mathbb{R}^{n+m+1}/\mathbb{Z}^{n+m+1}$. One can see directly that
\[
\varPsi^{-1}=\begin{pmatrix}I_{n} & \vline & 0 & \vline & \begin{matrix}0\\
\vdots\\
0
\end{matrix}\\
\hline 0 & \vline & I_{m} & \vline & \begin{matrix}0\\
\vdots\\
0
\end{matrix}\\
\hline \begin{matrix}\frac{-k_{1}}{\sqrt{2}+k_{0}} & \cdots & \frac{-k_{n}}{\sqrt{2}+k_{0}}\end{matrix} & \vline & \begin{matrix}0 & \cdots & 0\end{matrix} & \vline & \frac{1}{\sqrt{2}+k_{0}}
\end{pmatrix}
\]
so
\[
\varPsi^{-1}\tilde{y}_{i}=(0,\dots,\sqrt{2},\dots,0,\frac{\ell_{i}-k_{i}\sqrt{2}}{\sqrt{2}+k_{0}}),
\]
where $\sqrt{2}$ is in the $i$'th position and where we let $k_{i}\coloneqq0$
for $i=n+1,\dots,m$. Therefore, by Kronecker's theorem, there are
no $\lambda_{1},\dots,\lambda_{n+m+1}\in\mathbb{Q}$ that aren't all
zero and $r_{1},\dots,r_{n+m}\in\mathbb{Q}$ such that, for all $i=1,\dots,n+m$,
\begin{equation}
\lambda_{i}\sqrt{2}+\lambda_{n+m+1}\cdot\frac{\ell_{i}-k_{i}\sqrt{2}}{\sqrt{2}+k_{0}}+r_{i}=0.\label{eq:Kro}
\end{equation}
But by multiplying both sides by $\sqrt{2}+k_{0}$ we see that for
each $i$ this equation is equivalent to
\[
\left(2\lambda_{i}+r_{i}k_{0}+\lambda_{n+m+1}\ell_{i}\right)+\left(\lambda_{i}k_{0}-\lambda_{n+m+1}k_{i}+r_{i}\right)\sqrt{2}=0,
\]
which is -- since $\lambda_{i},r_{i},k_{i},\ell_{i}$ are all rational
-- equivalent to the two equations
\begin{align*}
\left(2\lambda_{i}+r_{i}k_{0}+\lambda_{n+m+1}\ell_{i}\right) & =0,\\
\left(\lambda_{i}k_{0}-\lambda_{n+m+1}k_{i}+r_{i}\right) & =0.
\end{align*}
These are $2n+2m$ homogeneous linear equations with $2n+2m+1$ variables
($\lambda_{1},\dots,\lambda_{n+m+1}$, $r_{1},\dots,r_{n+m}$), so
there must be a nontrivial solution. But if one the $r_{i}$'s is
nontrivial, one of the $\lambda_{i}$'s must be nontrivial as well
(by Equation (\ref{eq:Kro})). Contradiction!
\end{proof}
By \cite{CG18}, if $f:\mathbb{T}^{m}\to H$ is an epimorphism onto
a Hausdorff group $H$, then any number of generators of $H$ can
be lifted to $\mathbb{T}^{m}$. If we allow non-Hausdorff targets,
the bound of \propref{AbelianCase} is tight. 
\begin{lem}
\label{lem:GasStarofaTorus}For every $n\in\mathbb{N}$ there is an
epimorphism $f:\mathbb{T}^{n}\to H$ and generators $h_{1},\dots,h_{n-1}\in H$
of $H$ that cannot be lifted to generators of $\mathbb{T}^{n}$.
\end{lem}

\begin{proof}
Let $x_{i}\in\mathbb{T}^{n}$ be the element whose $i^{\text{th}}$
coordinate is $\sqrt{2}$ and the rest of whose coordinates are zero.
Set $\Delta=\left\langle x_{1},\dots,x_{n}\right\rangle $. Take $g_{i}=0$
for $i=1,\dots,n-1$. Then $\left\langle g_{1},\dots,g_{n-1}\right\rangle \Delta=\Delta$
is dense in $\mathbb{T}^{n}$, but there are no $\delta_{1},\dots,\delta_{n-1}\in\Delta$
such that $g_{1}+\delta_{1},\dots,g_{n-1}+\delta_{n-1}$ generate
$\mathbb{T}^{n}$.
\end{proof}

\section{Semisimple Groups\label{sec:SS}}

\subsection{Preliminary Lemmas}

\begin{lem}
\label{lem:commutator-still-dense}Let $G$ be a topological group
whose commutator is dense. If $F\leqslant G$ is a dense subgroup,
then its commutator is also dense in $G$.
\end{lem}

\begin{proof}
Suppose $x\in G$ and $V\ni x$ is an open neighbourhood; then by
assumption there are $x_{1},\dots,x_{2n}\in G$ such that $[x_{1},x_{2}]\cdots[x_{2n-1},x_{2n}]\in V$,
which means there are open neighbourhoods $U_{i}$ of $x_{i}$ ($i=1,\dots,2n$)
such that $[U_{1},U_{2}]\cdots[U_{2n-1},U_{2n}]\subseteq V$, and
by assumption there are $f_{i}\in U_{i}\cap F$ ($i=1,\dots,2n$),
so $[f_{1},f_{2}]\cdots[f_{2n-1},f_{2n}]$ belongs to $F'$ and to
$V$.
\end{proof}

\begin{lem}
\label{lem:CentralQuotient}Let $G$ be a topological group whose
commutator is dense. Then every central subgroup of $G$ is non-generating;
in particular, if $f:G\to H$ is an open homomorphism whose kernel
is central, then $f$ is absolutely Gasch\"utz. 
\end{lem}

\begin{proof}
If $Z\leqslant G$ is central and $F\leqslant G$ is a (not-necessarily-closed)
subgroup such that $FZ$ is dense in $G$, then $[F,F]=[FZ,FZ]$ is
dense in $[G,G]$, which is dense in $G$; in particular, $F\supseteq[F,F]$
is dense in $G$. In other words, every central subgroup is non-generating. 
\end{proof}

We need a lemma from \cite{GM17}. It is stated for groups, but the
same proof works for Lie algebras as well.
\begin{defn}
Let $\prod_{i=1}^{n}S_{i}$ be a product of Lie algebras, and denote
by $p_{\ell}:\prod_{i=1}^{n}S_{i}\to S_{\ell}$ the projection. We
say a subset $R\subseteq\prod_{i=1}^{n}S_{i}$ is \emph{diagonally
embedded }if there are two distinct $i,j\in\left\{ 1,\dots,n\right\} $
and an isomorphism $\varphi:S_{i}\to S_{j}$ such that $p_{j}(x)=\varphi(p_{i}(x))$
for all $x\in R$.
\end{defn}

\begin{lem}[{\cite[Lemma 3.7]{GM17}}]
\label{lem:lem}Let $S=\prod_{i=1}^{n}S_{i}$ be a product of simple
Lie algebras and let $p_{i}:S\to S_{i}$ be the projection onto the
$i^{\text{th}}$ coordinate. Let $R$ be a Lie subalgebra of $S$
that is not diagonally embedded and that satisfies $p_{i}(R)=S_{i}$
for every $i$. Then $R=S$.
\end{lem}

We will use the following well known fact.
\begin{fact*}[{\cite[1.5, 1.6]{Hum95}}]
If $G$ is a connected simple Lie group and $g\in G$ is any element,
then the closure of the orbit of $g$ under the group of (topological)
automorphisms of $G$ is of measure zero.
\end{fact*}
\begin{cor}
\label{cor:aut-orbit}Let $G=G_{1}\times\cdots\times G_{n}$ be a
connected semisimple Lie group, $H$ a connected simple Lie group,
$h\in H$ an element. Then
\[
\overline{\left\{ (g_{i})_{i}\in G\middle|g_{i}=\varphi(h)\text{ for some \ensuremath{i} and some isomorphism }\varphi:H\to G_{i}\right\} }
\]
is of measure zero.
\end{cor}

\subsection{Lifting Generators}

In this section we show that connected semisimple Lie groups are Gasch\"utz.

For the sake of the proof, we call an identity neighbourhood $V$
in a Lie group $G$ a \emph{strongly Zassenhaus neighbourhood} if
$\log$ is a well defined diffeomorphism on it and if, whenever $(g_{1},g_{2})\in V\times V$
are such that $\log g_{1},\log g_{2}$ generate $\mathfrak{g}$, it
follows $g_{1},g_{2}$ topologically generate $G$. Every connected
semisimple Lie group admits a strongly Zassenhaus neighbourhood by
\cite[Theorem 2.1]{BG03}. Moreover, by \cite[Theorem 2.4]{BG03}
the subset of pairs $(g_{1},g_{2})\in V^{2}$ which do not topologically
generate $G$ is contained in a closed null subset.
\begin{lem}
\label{lem:semi-simple-main-lemma}Let $G_{1},\dots,G_{m},H$ be connected
centre-free (nontrivial) simple Lie groups, and let $h_{1},\dots,h_{n}\in H$
be topological generators of $H$ (in particular, $n\geq2$). Then
there are $g_{1},g_{2}\in G\coloneqq\prod_{i=1}^{m}G_{i}$ such that
\[
(g_{1},h_{1}),(g_{2},h_{2}),(1_{G},h_{3}),\dots,(1_{G},h_{n})\in G\times H
\]
topologically generate $G\times H$.
\end{lem}

\begin{proof}
Let $\Omega\subseteq H,W\subseteq G$ be strongly Zassenhaus neighbourhoods,
and assume without loss of generality that $W\times\Omega$ is a strongly
Zassenhaus neighbourhood in $G\times H$ as well.

Since the subset of pairs $(y_{1},y_{2})\in\Omega^{2}$ which do not
generate $H$ is contained in a closed null subset, there are open
subsets $V_{1},V_{2}\subseteq H$ such that every pair in $V_{1}\times V_{2}$
generates $H$. Since $h_{1},\dots,h_{n}$ generate $H$, there are
words $w_{1},w_{2}$ in $n$ letters such that $w_{i}(\underline{h})\in V_{i}$
for $i=1,2$. 

In fact, we can control the sum of powers of $h_{j}$ in $w_{i}(h_{1},\dots,h_{n})$
for every $j=1,\dots,n$ and $i=1,2$. To see this, think of words
as elements in the free groups $F_{n}$ with basis $x_{1},\dots,x_{n}$.
For $i=1,2$ and $j=1,\dots,n$, we can multiply $w_{i}$ from the
right by $x_{j}\tilde{w}$, where $\tilde{w}\in[F_{n},F_{n}]$ is
a word such that $\tilde{w}(\underline{h})\in x_{j}^{-1}w_{i}(\underline{h})^{-1}V_{i}$.
This is possible since the commutator subgroup of $H$ is $H$, and
therefore if $\varphi:F_{n}\to H$ is a homomorphism with a dense
image, it follows $\varphi([F_{n},F_{n}])$ is dense as well (\lemref{commutator-still-dense}).
The new word (i.e.~$w_{i}x_{j}\tilde{w}$) still satisfies the requirement
above (i.e.~$w_{i}x_{j}\tilde{w}(\underline{h})\in V_{i}$). By repeating
this process, we can modify the sum of powers as we please. For reasons
that will be clear momentarily, let us assume the following: 
\begin{enumerate}
\item In $w_{1}$, the sum of powers of $x_{1}$ is $1$ and of $x_{2}$
is zero.
\item In $w_{2}$, the sum of powers of $x_{2}$ is $1$ and of $x_{1}$
is zero.
\end{enumerate}
Now, consider the map 
\begin{align*}
\Psi:G\times G & \to G\times G\\
(g_{1},g_{2}) & \mapsto(w_{1}(g_{1},g_{2},1_{G},\dots,1_{G}),w_{2}(g_{1},g_{2},1_{G},\dots,1_{G})).
\end{align*}
The differential of this map is 
\begin{align*}
\mathfrak{g}\oplus\mathfrak{g} & \to\mathfrak{g}\oplus\mathfrak{g}\\
(X_{1},X_{2}) & \mapsto\left(n_{1}^{1}X_{1}+n_{2}^{1}X_{2},n_{1}^{2}X_{1}+n_{2}^{2}X_{2}\right),
\end{align*}
where $n_{j}^{i}$ is the sum of powers of $x_{j}$ in $w_{i}$. By
our assumption, this implies that
\[
d\Psi(X_{1},X_{2})=(X_{1},X_{2}).
\]
In particular, the differential of $\Psi$ (at $(1_{G},1_{G})$) is
of full rank, and hence the image of $\Psi$ contains an open neighbourhood
of $(1_{G},1_{G})$. Let $U$ be the intersection of this identity
neighbourhood with $W$, and let $U_{1},U_{2}\subseteq G$ be open
subsets such that $U_{1}\times U_{2}\subseteq U$. As we remarked
above, the subset $\mathcal{N}$ of $U_{1}\times U_{2}$ of pairs
whose logarithms do not generate the Lie algebra of $G$ is contained
a closed null subset. 

Now, denote by $p_{\ell}:G\to G_{\ell}$ the projection onto the $\ell^{\text{th}}$
coordinate, for $\ell=1,\dots,m$. By \corref{aut-orbit}, the closure
of the subset
\[
\mathcal{B}\coloneqq\left\{ (y_{1},y_{2})\in G^{2}|\begin{array}{l}
\text{for some }\ell\text{ there is an isomorphism }\varphi:H\to G_{\ell}\\
\text{ such that }p_{\ell}(y_{i})=\varphi(w_{i}(\underline{h}))\text{ for }i=1,2
\end{array}\right\} 
\]
is null. Therefore we can find open subsets $\hat{U}_{i}\subseteq U_{i}$
($i=1,2$) such that 
\[
\hat{U}_{1}\times\hat{U}_{2}\subseteq(U_{1}\times U_{2})\backslash\left(\mathcal{N}\cup\mathcal{B}\right).
\]

Let $g_{1},g_{2}\in G$ be such that $w_{i}(g_{1},g_{2},1_{G},\dots,1_{G})\in\hat{U}_{i}$
for $i=1,2$. Such an $n$-tuple actually exists since, by construction,
the subset $\hat{U}_{1}\times\hat{U}_{2}$ lies in the image of $\Psi$.
We claim $(g_{1},h_{1}),(g_{2},h_{2}),(1_{G},h_{3}),\dots,(1_{G},h_{n})\in G\times H$
generate $G\times H$. Set $\underline{g}=(g_{1},g_{2},1_{G},\dots,1_{G})$.
Clearly, it is enough to show the pair 
\begin{align*}
\gamma_{1}\coloneqq(w_{1}(\underline{g}),w_{1}(\underline{h})) & =w_{1}((g_{1},h_{1}),(g_{2},h_{2}),(1_{G},h_{3}),\dots,(1_{G},h_{n})),\\
\gamma_{2}\coloneqq(w_{2}(\underline{g}),w_{2}(\underline{h})) & =w_{2}((g_{1},h_{1}),(g_{2},h_{2}),(1_{G},h_{3}),\dots,(1_{G},h_{n}))
\end{align*}
 generates $G\times H$. The Lie algebra of $G\times H$ is $\mathfrak{g}\oplus\mathfrak{h}=\mathfrak{g}_{1}\oplus\cdots\oplus\mathfrak{g}_{n}\oplus\mathfrak{h}$;
denote by $R$ the Lie subalgebra generated by $\log\gamma_{1},\log\gamma_{2}$.
By our construction, $R$ projects onto $\mathfrak{h}$ and $\mathfrak{g}$,
and therefore also onto each $\mathfrak{g}_{\ell}$ (for $\ell=1,\dots,m$).
The subset $\left\{ \log\gamma_{1},\log\gamma_{2}\right\} $ is not
diagonally embedded: if it were, the isomorphism would be the differential
of an isomorphism between the Lie groups (since the groups are connected,
centre-free and semisimple), contrary to the construction of $\gamma_{1},\gamma_{2}$.
Therefore, $R$ is not diagonally embedded, and hence $R=\mathfrak{g}\oplus\mathfrak{h}$
by \lemref{lem}. Since $\gamma_{1},\gamma_{2}\in W\times\Omega$
and $W\times\Omega$ is a strongly Zassenhaus neighbourhood, this
means they generate $G\times H$. 
\end{proof}
\begin{cor}
\label{cor:SsCase}Connected semisimple Lie groups are Gasch\"utz. In
other words, if $G$ is a connected semisimple Lie group, $f:G\to H$
is an open epimorphism and $h_{1},\dots,h_{n}\in H$ are generators
with $n\geqslant\tr G$, then there are generating lifts $g_{1},\dots,g_{n}\in G$.
\end{cor}

\begin{rem}
Recall that $\tr G=2$ if $G$ is nontrivial, and that $H$ is necessarily
a connected semisimple Lie group.
\end{rem}

\begin{proof}
For centre-free groups this follows immediately from the lemma, since
connected centre-free semisimple Lie groups are direct products of
connected centre-free simple Lie groups, and their only normal subgroups
are subproducts, and hence every epimorphism from such a group is
a composition of projections as in the lemma. 

Now, let $G$ be a general semisimple Lie group, let $f:G\to H$ be
an epimorphism and $\underline{h}\in H^{n}$ a generating system with
$n\geqslant\tr G$. Denote $K=\ker f$, $Z=Z(G)$, $\hat{G}=G/Z(G)$,
$\hat{H}=H/f(Z)$. Observe $f(Z)\subseteq Z(H)$ is discrete and hence
closed. We have a natural epimorphism $\hat{f}:\hat{G}\to\hat{H}$
and a commutative diagram
\[
\xymatrix{G\ar[r]^{f}\ar[d] & H\ar[d]\\
\hat{G}\ar[r]_{\hat{f}} & \hat{H}
}
\]
Clearly, $h_{1}f(Z),\dots,h_{n}f(Z)$ generate $\hat{H}$, so there
are $g_{1}Z,\dots,g_{n}Z$ which generate $G/Z$ and such that $\hat{f}(g_{i}Z)=h_{i}f(Z)$.
By \lemref{CentralQuotient}, $g_{1},\dots,g_{n}$ generate $G$.
It is still possible that they are not lifts of $h_{1},\dots,h_{n}$.
However, the fact $\hat{f}(g_{i}Z)=h_{i}f(Z)$ means that $f(g_{i})^{-1}h_{i}\in f(Z)$
for $i=1,\dots,n$. Thus, there are $\varepsilon_{1},\dots,\varepsilon_{n}\in Z$
such that $f(\varepsilon_{i})=f(g_{i})^{-1}h_{i}$, so $f(g_{i}\varepsilon_{i})=h_{i}$
for $i=1,\dots,n$. Moreover, $g_{1}\varepsilon_{1},\dots,g_{n}\varepsilon_{n}$
generate $G$ (again by \lemref{CentralQuotient}, because they are
still lifts of $g_{1}Z,\dots,g_{n}Z$), so we are done.
\end{proof}

\section{Reductive Groups\label{sec:Reductive}}

In this section we consider products of semisimple groups with abelian
groups. Generation in such groups is simple:
\begin{lem}[{\cite[Lemma 6.6]{AN24}}]
\label{lem:GenReductive}Let $S$ be a connected semisimple Lie group
and $A$ a connected abelian Lie group, and denote by $p_{1}:S\times A\to S$,
$p_{2}:S\times A\to A$ the projections. Then $(s_{1},a_{1}),\dots,(s_{n},a_{n})\in S\times A$
generate $S\times A$ if and only if $s_{1},\dots,s_{n}\in S$ generate
$S$ and $a_{1},\dots,a_{n}\in A$ generate $A$. In particular, $\tr{S\times A}=\max\{\tr S,\tr A\}$.
\end{lem}

\begin{prop}
\label{prop:Gas-Red}Let $S$ be a connected semisimple Lie group
and $A$ a connected abelian Lie group. Let $f:S\times A\to H$ be
an open epimorphism, and let $f^{\mathrm{ab}}:A\to H/f(S)$ be the
induced epimorphism. If $h_{1},\dots,h_{n}\in H$ are generators with
$n\geqslant\tr{S\times A}$, then they admit generating lifts to $G$
via $f$ if and only if $h_{1}f(S),\dots,h_{n}f(S)$ admit generating
lifts to $A$ via $f^{\mathrm{ab}}$.

If $S$ has finite centre and $H$ is Hausdorff, then $H/f(S)$ is
Hausdorff as well.
\end{prop}

\begin{proof}
It is easy to see that, if $g_{1},\dots,g_{n}\in S\times A$ are generating
lifts of $h_{1},\dots,h_{n}$, then their images in $A$ are generating
lifts of $h_{1}f(S),\dots,h_{n}f(S)$. Thus, let $h_{1},\dots,h_{n}\in H$
be generators with $n\geqslant\tr{S\times A}$ such that $h_{1}f(S),\dots,h_{n}f(S)$
admit generating lifts to $A$ via $f^{\mathrm{ab}}$. We will find
generating lifts to $S\times A$.

Set $G=S\times A$ and let $\Delta$ be the kernel of $f$. Let $g_{1},\dots,g_{n}\in G$
be arbitrary lifts of $h_{1},\dots,h_{n}$. Since $h_{1},\dots,h_{n}$
generate $H$, we get that $\left\langle g_{1},\dots,g_{n}\right\rangle \Delta$
is dense in $G$; we need to prove there are $\delta_{1},\dots,\delta_{n}\in\Delta$
such that $g_{1}\delta_{1},\dots,g_{n}\delta_{n}$ generate $G$. 

We first prove the statement under the assumption that either $\Delta\subseteq A$
or $\Delta\subseteq S$. If $\Delta\subseteq A$, then $(S\times A)/\Delta=S\times(A/\Delta)$.
There are $s_{i}\in S,a_{i}\in A$ such that $h_{i}=(s_{i},a_{i}\Delta)$
for $i=1,\dots,n$. Since $a_{1}\Delta,\dots,a_{n}\Delta$ generate
$A/\Delta$, there are $\varepsilon_{1},\dots,\varepsilon_{n}\in\Delta$
such that $a_{1}\varepsilon_{1},\dots,a_{n}\varepsilon_{n}$ generate
$A$ (since the map $A\longrightarrow A/\Delta$ is exactly $f^{\mathrm{ab}}$),
so $(s_{1},a_{1}\varepsilon_{1}),\dots,(s_{n},a_{n}\varepsilon_{n})$
are lifts of $(s_{1},a_{1}\Delta),\dots,(s_{n},a_{n}\Delta)$ that
generate $S\times A$ (by \lemref{GenReductive}), as needed. The
case $\Delta\subseteq S$ is exactly the same, only naturally one
needs to use the fact $n\geqslant\tr S$ and \corref{SsCase}.

Next, we prove the statement under the assumption that $\Delta$ is
discrete. The centre of $S$ is non-generating as a subgroup of $S$
(\lemref{CentralQuotient}), and hence as a subgroup of $S\times A$
(by \lemref{GenReductive}). Denote it by $Z$; we have the following
commutative diagram:
\[
\xymatrix{S\times A\ar[r]^{f}\ar[d] & H\ar[d]\\
(S/Z\times A)\ar[r]_{\tilde{f}} & H/f(Z)
}
\]
where $\tilde{f}:(S/Z\times A)\to H/f(Z)$ is the epimorphism induced
by $f$. By \lemref{AbsGas}, the elements $h_{1},\dots,h_{n}\in H$
admit generating lifts to $S\times A$ if and only if the images of
$h_{1},\dots,h_{n}$ in $H/f(Z)$ admit generating lifts to $(S/Z\times A)$
via $\tilde{f}$, so we may have assumed $S$ is centre-free. If
$S$ is centre-free and $\Delta$ is discrete, $\Delta$ is contained
in $A$ (being discrete and normal, hence central), so we may proceed
as above. 

We can now prove the statement in general. Set $N=\overline{\Delta}$
and consider $\conn N$. Since $\left\langle g_{1},\dots,g_{n}\right\rangle \Delta$
is dense in $G$, we get that $g_{1}N,\dots,g_{n}N$ generate $G/N$,
which (by the isomorphism $G/N\cong(G/\conn N)/(N/\conn N)$) is the
same thing as saying $(g_{1}\conn N)N/\conn N,\dots,(g_{n}\conn N)N/\conn N$
generate $(G/\conn N)/(N/\conn N)$. Since $G/\conn N$ is still a
direct product of a semisimple Lie group and an abelian Lie group,
and $N/\conn N$ is a closed normal subgroup, we may proceed as in
the case $\Delta$ is discrete. This means that there are $\varepsilon_{1}\conn N,\dots,\varepsilon_{n}\conn N\in N/\conn N$
such that $g_{1}\varepsilon_{1}\conn N,\dots,g_{n}\varepsilon_{n}\conn N$
generate $G/\conn N$. Now, $\Delta$ is dense in $N$, hence its
image in $N/\conn N$ is dense in $N/\conn N$; but $N/\conn N$ is
discrete, so it means $\Delta$ surjects onto it, so we may assume
$\varepsilon_{1},\dots,\varepsilon_{n}\in\Delta$. Thus, $g_{1}\varepsilon_{1}\conn N,\dots,g_{n}\varepsilon_{n}\conn N\in G/\conn N$
generate $G/\conn N$, and we need to prove there are $\delta_{1},\dots,\delta_{n}\in\Delta$
such that $g_{1}\varepsilon_{1}\delta_{1},\dots,g_{n}\varepsilon_{n}\delta_{n}$
generate $G$.

In order to simplify notations, let us just assume $N$ is connected.
So we are now in the following scenario: $G=S\times A$ where $S$
is a connected semisimple Lie group and $A$ is a connected abelian
Lie group, $\Delta\trianglelefteqslant G$ is a not-necessarily-closed
normal subgroup such that $N=\overline{\Delta}$ is connected, and
$g_{1},\dots,g_{n}\in G$ are elements such that $g_{1}N,\dots,g_{n}N$
generate $G/N$. We need to prove there are $\delta_{1},\dots,\delta_{n}\in\Delta$
such that $g_{1}\delta_{1},\dots,g_{n}\delta_{n}$ generate $G$.

Denote by $p_{S}:S\times A\to S$ the projection, and consider $p_{S}(\Delta)$.
This is a normal subgroup of $S$, hence a closed subgroup (see for
example \cite{Rag72}). We have $p_{S}(N)=p_{S}(\overline{\Delta})\subseteq\overline{p_{S}(\Delta)}=p_{S}(\Delta)$,
and clearly $p_{S}(\Delta)\subseteq p_{S}(N)$, so $p_{S}(N)=p_{S}(\Delta)$.
Denote $N_{S}=p_{S}(N)=p_{S}(\Delta)$. This is a connected normal
subgroup of a semisimple Lie group, hence also a connected semisimple
Lie group, which means in particular that $[N_{S},N_{S}]=N_{S}$.
Now, consider $[\Delta,\Delta]$; on the one hand,
\[
p_{S}([\Delta,\Delta])=[p_{S}(\Delta),p_{S}(\Delta)]=[N_{S},N_{S}]=N_{S}.
\]
On the other hand, if $p_{A}:S\times A\to A$ is the projection onto
$A$, we clearly have that $p_{A}([\Delta,\Delta])\leqslant A$ is
trivial. In other words, $[\Delta,\Delta]$ is contained in $S$,
which means that $[\Delta,\Delta]=N_{S}$ (so, in particular, it is
closed). Clearly $[\Delta,\Delta]\subseteq\Delta$, so $N_{S}=p_{S}(\Delta)\subseteq\Delta$.
That is, $\Delta\trianglelefteqslant S\times A$ contains its own
projection into $S$, so $\Delta=p_{S}(\Delta)\times p_{A}(\Delta)=N_{S}\times p_{A}(\Delta)$.
In particular, this means that $N=N_{S}\times\overline{p_{A}(\Delta)}$. 

The rest follows from the case $\Delta$ is contained in either $A$
or $S$, since dividing by $N_{S}\times\rho_{A}(\Delta)$ is the same
thing as first dividing by $N_{S}$ and then by (the image of) $\rho_{A}(\Delta)$.

\subparagraph{$ $}

Observe that, if $\Delta=N_{S}\times p_{A}(\Delta)$ is closed, $p_{A}(\Delta)$
must be closed as well, so $\ker f^{\mathrm{ab}}=p_{A}(\Delta)$ is
closed. For this part, we need to assume the centre of $S$ is finite,
because otherwise we might have changed the closedness of $\Delta\cdot S$
when dividing by it above.
\end{proof}
\begin{rem}
If $S$ has infinite centre, it is possible for $H/f(S)$ not to be
Hausdorff even though $H$ is.
\end{rem}

\section{Abels--Noskov Groups\label{sec:AN}}

At this point, we have handled the cases where $G$ is semisimple,
abelian, or a product of such groups. Extending this to the general
case involves using the structure theory of Lie groups. Fortunately,
we can rely on the work of \cite{AN24} where the Frattini subgroup
of a connected Lie group is characterised. Elements of this subgroup
are irrelevant to generation problems, and by quotienting out this
subgroup, we obtain a group with a highly structured form, which we
refer to as \textit{Abels--Noskov groups}.

This section is dedicated to establishing Theorem \ref{thm:intro-main}
for this specific class. In the following section, we will show how
the proof of the general case (Theorem \ref{thm:intro-main}) can
be reduced to this setting.

\subsection{Generating Sets of Modules}

We start by collecting facts about finite dimensional representations
of abstract groups, and their generating sets. 

Let $L$ be a group, $V$ a finite dimensional vector space over $\mathbb{R}$
and $\rho:L\to\mathrm{GL}(V)$ a representation. We denote by $\mathbb{R}[\rho]$
the subalgebra of $\mathrm{End}_{\mathbb{R}}(V)$ generated by $\rho(L)$.
We call $V$ an $L$-module, and we call a subspace $U$ of $V$ a
\emph{submodule }if it is $L$-invariant (equivalently, $\mathbb{R}[\rho]$-invariant).
In this case, we denote by $\rho|_{U}:L\to\mathrm{GL}(U)$ the corresponding
representation. We say $V$ (or $\rho$) is \emph{irreducible }if
it admits no nontrivial submodules. It is basic fact in representation
theory that $V$ is a direct sum of irreducible submodules if, and
only if, every submodule admits a complement (i.e., every submodule
$U\subseteq V$ admits a submodule $W\subseteq V$ such that $V=U\oplus W$
as vector spaces, and hence as $L$-modules). In this case we say
that $V$ is \emph{completely reducible}. For a subset $S\subseteq V$,
the \emph{submodule generated by $S$ }is the smallest submodule of
$V$ containing $S$.

We say that $V$ is \emph{isotypic }if it is a direct sum of irreducible
representations which are all isomorphic to one another. The decomposition
is not unique, but obviously the number of irreducible submodules
in the sum is; it is called the \emph{multiplicity }of the isotypic
module $V$. We denote it by $m(V)$.

Every completely reducible module is a direct sum of isotypic submodules,
and this decomposition is unique. Given an irreducible representation
$\sigma$ of $L$, we denote by $V_{\sigma}$ the isotypic component
of $V$ corresponding to $\sigma$ (which might be zero). The \emph{$\sigma$-multiplicity
}of $V$ is the multiplicity of $V_{\sigma}$. We denote it by $m_{\sigma}(V)$.
We denote $m(V)=\max_{\sigma}\left\{ m_{\sigma}(V)\right\} $.

If $\sigma:L\to\mathrm{GL}(V)$ is an irreducible representation then
by Schur's lemma, the centraliser of $\mathbb{R}[\sigma]$ inside
$\mathrm{End}_{\R}(V)$ is a division algebra $k$ over $\mathbb{R}$,
and as such $k$ is isomorphic to $\mathbb{R}$, $\mathbb{C}$ or
$\mathbb{H}$ (the quaternions). In this case we may regard $V$ as
a vector space over $k$, and it holds that $\mathbb{R}[\sigma]\cong\mathrm{End}_{k}(V)$
(see \cite[Proposition C.1]{AN24}). If $\rho$ is isotypic of type
$\sigma$, then the diagonal embedding gives us an isomorphism $\mathbb{R}[\sigma]\to\mathbb{R}[\rho]$.
More generally, if $\rho$ is completely reducible then $\R[\rho]$
is a direct sum of $\R[\sigma]$, one for each irreducible $\sigma$
that occurs with nonzero multiplicity in $\rho$. 

Given a representation $\rho:L\to\mathrm{GL}(V)$, we get another
module, namely $\mathbb{R}[\rho]$ itself, where the action of $L$
on $\mathbb{R}[\rho]$ is given by $\ell.A=\rho(\ell)\circ A\in\mathbb{R}[\rho]$.
If $\rho$ is irreducible, then $\mathbb{R}[\rho]$ is isotypic of
type $\rho$, and its multiplicity is $\dim_{k}V$.
\begin{lem}[{\cite[Corollary C.2]{AN24}}]
\label{lem:gen-module}Let $\sigma$ be an irreducible representation
of a group $L$, and let $\rho:L\to\mathrm{GL}(V)$ be an isotypic
representation of $L$ of type $\sigma$ on a finite dimensional real
vector space $V$, which is isomorphic to the sum of $\ell$ copies
of $\sigma$. Denote by $\mathrm{d}_{\mathbb{R}[\rho]}(V)$ the minimal
number of elements needed to generate $V$ as an $\mathbb{R}[\rho]$-module.
Using the notations above, we have
\[
\mathrm{d}_{\mathbb{R}[\rho]}(V)=\left\lceil \frac{\ell}{\dim_{k}\sigma}\right\rceil =\left\lceil \frac{\dim_{\mathbb{R}}V}{\dim_{\mathbb{R}}\mathbb{R}[\rho]}\right\rceil .
\]
\end{lem}

\begin{cor}[{\cite[Corollary C.3]{AN24}}]
\label{cor:gen-module}Let $\rho:L\to\mathrm{GL}(V)$ be a completely
reducible finite dimensional representation of a group $L$ on real
vector space $\mathbb{R}$ with isotypic components $\left\{ V_{\sigma}\right\} _{\sigma}$.
Let $k_{\sigma}$ be the centraliser of $\mathbb{R}[\rho|_{V_{\sigma}}]$,
so that $V_{\sigma}$ is a vector space over the division algebra
$k_{\sigma}$. Denote by $\mathrm{d}_{\mathbb{R}[\rho]}(V)$ the minimal
number of elements needed to generate $V$ as an $\mathbb{R}[\rho]$-module.
Using the notations above, we have
\[
\mathrm{d}_{\mathbb{R}[\rho]}(V)=\max_{\sigma}\left\{ \left\lceil \frac{m_{\sigma}(V)}{\dim_{k_{\sigma}}\sigma}\right\rceil \right\} =\max_{\sigma}\mathrm{d}_{\mathbb{R}[\rho|_{V_{\sigma}}]}(V_{\sigma}).
\]
\end{cor}

\subsection{Generating Sets of Abels--Noskov Groups}
\begin{defn}
\label{def:abels-noskov}An \emph{Abels--Noskov group} is a connected
Lie group of the form $(S\times A)\ltimes_{\rho}V$, where
\begin{enumerate}
\item $S$ is a connected semisimple Lie group with finite centre;
\item $A$ is a connected abelian Lie group;
\item $V$ is a finite dimensional real vector space;
\item The semi-direct product is defined via a representation $\rho:(S\times A)\to\mathrm{GL}(V)$
that is completely reducible and admits no non-trivial fixed vectors. 
\end{enumerate}
\end{defn}

We may now quote the beautiful result of Abels and Noskov which characterizes
generating sets.
\begin{thm}[{\cite[Lemmas 6.4 and 6.5]{AN24}}]
\label{thm:gen-of-AN}Let $L\ltimes_{\rho}V$ be an Abels--Noskov
group and let $(\ell_{1},v_{1}),\dots,(\ell_{n},v_{n})\in L\ltimes_{\rho}V$.
Let $\varepsilon_{\boldsymbol{\ell}}:\mathbb{R}[\rho]^{n}\to\mathbb{R}[\rho]$
be the $\mathbb{R}[\rho]$-module homomorphism 
\[
\varepsilon_{\boldsymbol{\ell}}(a_{1},\dots,a_{n})=\sum_{i=1}^{n}a_{i}(\bbm 1-\rho(\ell_{i}))
\]
and let $\alpha_{\boldsymbol{v}}:\mathbb{R}[\rho]^{n}\to V$ be the
$\mathbb{R}[\rho]$-module homomorphism
\[
\alpha_{\boldsymbol{v}}(a_{1},\dots,a_{n})=\sum_{i=1}^{n}a_{i}v_{i}.
\]
Then $(\ell_{1},v_{1}),\dots,(\ell_{n},v_{n})$ generate $L\ltimes_{\rho}V$
if and only if the following two conditions hold: $\ell_{1},\dots,\ell_{n}$
generate $L$, and 
\[
\alpha_{\boldsymbol{v}}(\ker\varepsilon_{\boldsymbol{\ell}})=V.
\]
\end{thm}

Observe that every $\mathbb{R}[\rho]$-module homomorphism $\alpha:\mathbb{R}[\rho]^{n}\to V$
is of the form $\alpha_{\boldsymbol{v}}$ for $\boldsymbol{v}=(\alpha(e_{1}),\dots,\alpha(e_{n}))$,
where $e_{i}\in\mathbb{R}[\rho]^{n}$ is $\bbm 1$ in the $i^{\text{th}}$
coordinate and $0$ in the others. If $\ell_{1},\dots,\ell_{n}$ generate
$L$, then $\varepsilon_{\boldsymbol{\ell}}$ is always surjective,
by \cite[Lemma C.4]{AN24}. 
\begin{thm}[{\cite[Theorem 6.3]{AN24}}]
\label{thm:AN-tr}Let $G=L\ltimes_{\rho}V$ be an Abels--Noskov
group. Then
\[
\tr G=\max\left\{ \tr L,\mathrm{d}_{\mathbb{R}[\rho]}(V)+1\right\} 
\]
\end{thm}

\subsection{The Case \texorpdfstring{$\ker f\subseteq V$}{the Kernel is Contained in V}}
\begin{lem}
\label{lem:RedToIsotypic}Let $G=L\ltimes_{\rho}V$ be an Abels--Noskov
group, let $\left\{ V_{\sigma}\right\} _{\sigma}$ be the isotypic
components of $V$, and let $\pi_{\sigma}:L\ltimes_{\rho}V\to L\ltimes_{\rho|_{V_{\sigma}}}V_{\sigma}$
be the projections. Let $X\subseteq G$. Then $X$ generates $G$
if and only if $\pi_{\sigma}(X)$ generates $L\ltimes_{\rho|_{V_{\sigma}}}V_{\sigma}$
for every $\sigma$.
\end{lem}

\begin{proof}
First, recall that if $U\subseteq V$ is a submodule, then the isotypic
components of $U$ are $U\cap V_{\sigma}$, and that $U\cap V_{\sigma}=p_{\sigma}(U)$
(where $p_{\sigma}:V\to V_{\sigma}$ is the projection). In particular,
if $U$ projects onto every isotypic component of $V$, then $U=V$.
With this in mind, the lemma follows immediately from \thmref{gen-of-AN}.
\end{proof}
\begin{lem}
Let $L$ be a path-connected topological group and let $\rho:L\to\mathrm{End}(V)$
be completely reducible representation of $L$ on a finite dimensional
real vector space $V$ without nonzero fixed vectors. Let $U\leqslant V$
be a (not-necessarily-closed) subgroup of (the additive group of)
$V$. If $U$ is $L$-invariant, then it is a submodule (and in particular
closed).
\end{lem}

\begin{proof}
First, denote by $U^{p}$ the path-connected component of the identity
of $U$. Recall that $U^{p}$ is a characteristic subgroup of $U$.
Since $U^{p}$ is path-connected, it is a connected Lie subgroup of
$V$ (see, e.g., \cite{Got69}), and hence a vector subspace. Since
vector subspaces are always closed (in finite dimension), we get that
$U^{p}$ is a closed subgroup. Since $U$ is $L$-invariant and $U^{p}$
is characteristic in $U$, we get that $U^{p}$ is $L$-invariant;
in other words, $U^{p}$ is a submodule. We will show that $U=U^{p}$.

Since $\rho$ is a completely reducible representation, there is a
submodule $W$ complementing $U^{p}$, so that $V=U^{p}\oplus W$
as $L$-modules. Since both $U$ and $W$ are $L$-invariant, the
intersection $U\cap W$ is $L$-invariant as well. The path-component
of the identity of $U\cap W$ is $\left\{ 0\right\} $, so all the
path-components of $U\cap W$ are singletons; since $L$ is path-connected,
it follows $L$ acts trivially on $U\cap W$. By assumption, the action
of $L$ on $V$ does not admit nonzero fixed vectors; therefore, $U\cap W=\left\{ 0\right\} $.
Since $U$ contains $U^{p}$ and $V=U^{p}\oplus W$, we get that $U=U^{p}\oplus(U\cap W)=U^{p}$,
as needed.
\end{proof}
\begin{lem}
\label{lem:submodule-case}Let $G=L\ltimes_{\rho}V$ be an Abels--Noskov
group, and let $f:G\to H$ be an open epimorphism such that $\ker f$
is contained in $V$. If $h_{1},\dots,h_{n}\in H$ generate $H$ and
$n\geqslant\tr G$ then there are lifts $g_{1},\dots,g_{n}\in G$
that generate $G$.
\end{lem}

\begin{proof}
Write $U=\ker f$; since it is a normal subgroup of $G$ contained
in $V$, it is $L$-invariant, and hence an $L$-submodule of $V$
by the previous lemma. Therefore, there is an $L$-submodule $W\subseteq V$
such that $V=W\oplus U$, and $H$ is naturally isomorphic to $L\ltimes_{\rho|_{W}}W$. 

In light of \lemref{RedToIsotypic}, we may assume $\rho$ is an isotypic
representation of type $\sigma$. As usual, we denote by $\mathbb{R}[\rho]$
the subalgebra of $\mathrm{End}(V)$ generated by $\rho(L)$. It is
isomorphic to $\mathbb{R}[\rho|_{W}]$, the subalgebra of $\mathrm{End}(W)$
generated by $\rho|_{W}(L)$ (except in the case $U=V$, in which
case everything below works with minor changes). Writing $h_{i}=(\ell_{i},w_{i})$
for $i=1,\dots,n$, we denote by $\varepsilon_{\boldsymbol{\ell}}:\mathbb{R}[\rho]^{n}\to\mathbb{R}[\rho]$
the map 
\[
\varepsilon_{\boldsymbol{\ell}}\left(a_{1},\dots,a_{n}\right)=\sum_{i=1}^{n}a_{i}(\bbm 1-\rho(\ell_{i})),
\]
as above. Abusing notations somewhat, we identify $\mathbb{R}[\rho]$
with $\mathbb{R}[\rho|_{W}]$ and denote by $\varepsilon_{\boldsymbol{\ell}}$
also the map $\mathbb{R}[\rho|_{W}]^{n}\to\mathbb{R}[\rho|_{W}]$
defined by
\[
\varepsilon_{\boldsymbol{\ell}}\left(a_{1},\dots,a_{n}\right)=\sum_{i=1}^{n}a_{i}(\bbm 1-\rho|_{W}(\ell_{i})).
\]
We write $P=\ker\varepsilon_{\boldsymbol{\ell}}$. We denote by $\alpha_{\boldsymbol{w}}:\mathbb{R}[\rho]^{n}\to W$
the map 
\[
\alpha_{\boldsymbol{w}}\left(a_{1},\dots,a_{n}\right)=\sum_{i=1}^{n}a_{i}w_{i}.
\]
The fact $h_{1},\dots,h_{n}$ generate $H$ means that $\alpha_{\boldsymbol{w}}(P)=W$. 

We need to find lifts $g_{1},\dots,g_{n}\in G$ that generate $G$.
This means finding $u_{1},\dots,u_{n}\in U$ such that $(\ell_{1},w_{1}\oplus u_{1}),\dots,(\ell_{n},w_{n}\oplus u_{n})$
generate $G$. For $\boldsymbol{u}=(u_{1},\dots,u_{n})\in U^{n}$,
we denote by $\alpha_{\boldsymbol{u}}:\mathbb{R}[\rho]^{n}\to U$
the map 
\[
\alpha_{\boldsymbol{u}}\left(a_{1},\dots,a_{n}\right)=\sum_{i=1}^{n}a_{i}u_{i}.
\]
We denote by $\alpha_{\boldsymbol{w}\oplus\boldsymbol{u}}:\mathbb{R}[\rho]^{n}\to V$
the map 
\[
\alpha_{\boldsymbol{w}\oplus\boldsymbol{u}}\left(a_{1},\dots,a_{n}\right)=\sum_{i=1}^{n}a_{i}(w_{i}\oplus u_{i}),
\]
so that $\alpha_{\boldsymbol{w}\oplus\boldsymbol{u}}=\alpha_{\boldsymbol{w}}\oplus\alpha_{\boldsymbol{u}}$.
Therefore, we need to find $\boldsymbol{u}=(u_{1},\dots,u_{n})\in U^{n}$
such that $\left(\alpha_{\boldsymbol{w}}\oplus\alpha_{\boldsymbol{u}}\right)(P)=V$.
As noted above, every map $\alpha:\mathbb{R}[\rho]^{n}\to U$ is of
the form $\alpha_{\boldsymbol{u}}$ for some $\boldsymbol{u}=(u_{1},\dots,u_{n})\in U^{n}$
(namely, $\boldsymbol{u}=(\alpha(e_{1}),\dots,\alpha(e_{n}))$), so
we actually just need to find a map $\alpha:\mathbb{R}[\rho]^{n}\to U$
such that $\left(\alpha_{\boldsymbol{w}}\oplus\alpha\right)(P)=V$.
Since $\mathbb{R}[\rho]^{n}\cong P\oplus\mathbb{R}[\rho]$, this is
the same thing as finding a map $\alpha:P\to U$ such that $\left(\bar{\alpha}_{\boldsymbol{w}}\oplus\alpha\right)(P)=V$,
where $\bar{\alpha}_{\boldsymbol{w}}$ is the restriction of $\alpha_{\boldsymbol{w}}$
to $P$.

Write $Q=\ker\bar{\alpha}_{\boldsymbol{w}}$. Since $P$ is completely
reducible, we get that $P\cong Q\oplus W$. Observe that there is
a surjective $\mathbb{R}[\rho]$-map $\beta:Q\to U$ if (and only
if) $m(Q)\geqslant m(U)$. Recall that 
\[
m(\mathbb{R}[\rho])=\dim_{k}\sigma,
\]
where $k$ is the corresponding Schur field (i.e., the centraliser
of $\mathbb{R}[\rho]$). Thus, 
\[
m(P)=(n-1)\cdot m(\mathbb{R}[\rho])=(n-1)\dim_{k}\sigma.
\]
So we know that
\[
m(Q)=m(P)-m(W)=(n-1)\dim_{k}\sigma-m(W).
\]
So what we need is that  
\[
(n-1)\dim_{k}\sigma\geqslant m(W)+m(U)=m(V).
\]
That is, we need
\[
n\geqslant\frac{m(V)}{\dim_{k}\sigma}+1.
\]
Recall that $d_{\mathbb{R}[\rho]}(V)=\left\lceil \frac{m(V)}{\dim_{k}\sigma}\right\rceil $
(\lemref{gen-module}) and $\tr G\geqslant d_{\mathbb{R}[\rho]}(V)+1$
(\thmref{AN-tr}). Thus, 
\[
n\geqslant\tr G\geqslant d_{\mathbb{R}[\rho]}(V)+1=\left\lceil \frac{m(V)}{\dim_{k}\sigma}\right\rceil +1,
\]
so we're in the clear. This means there is such a $\beta$, and we
can define $\alpha:Q\oplus W\to U$ by $\alpha(q,w)=\beta(q)$. Observe
that, with the identification $P\cong Q\oplus W$, the map $\bar{\alpha}_{\boldsymbol{w}}:P\to W$
is the projection onto the second coordinate (it was, after all, exactly
the map $\bar{\alpha}_{\boldsymbol{w}}$ that gave us this isomorphism).
Therefore, $\left(\bar{\alpha}_{\boldsymbol{w}}\oplus\alpha\right)(W)=W$
and $\left(\bar{\alpha}_{\boldsymbol{w}}\oplus\alpha\right)(Q)=U$,
so $\left(\bar{\alpha}_{\boldsymbol{w}}\oplus\alpha\right)(P)=V$,
as needed.
\end{proof}

\subsection{The Case \texorpdfstring{$\ker f\cap V$}{the Kernel} is Trivial}
\begin{lem}
\label{lem:NormalOfSDP}Let $G=L\ltimes_{\rho}V$ be an Abels--Noskov
group, and let $K\trianglelefteqslant G$ be a not-necessarily-closed
normal subgroup of $G$ such that $K\cap V$ is trivial. Then $K\leqslant\ker\rho\leqslant L$
.
\end{lem}

\begin{proof}
Let $\pi:G\to L$ denote the quotient map. Since $V$ is abelian,
we have, for any $k\in K$ and $v\in V$, 
\[
[k,v]=[\pi(k),v]=\rho(\pi(k))v-v.
\]
Moreover, $[k,v]\in K\cap V$, since both $K$ and $V$ are normal.
As $K\cap V$ is trivial, it follows that $K$ and $\pi(K)$ act trivially
on $V$ by conjugation; since $\pi(K)$ is contained in $L$ (and
the action by conjugation is given by $\rho$), we have that $\pi(K)\leqslant\ker\rho$. 

As $K$ is normal in $G$, $\overline{\pi(K)}$ is normal in $L$.
Any connected simple normal subgroup of $S$ is either contained
in $\overline{\pi(K)}$ or else it commutes with $\overline{\pi(K)}$.
Let $H_{0}$ be the product all connected simple normal subgroups
of $S$ which are not contained in $\overline{\pi(K)}$, and set $H=H_{0}A$.
Then $\overline{\pi(K)}H=L$ and $[\overline{\pi(K)},H]=\{1\}$. Now,
let $k\in K$. There is some $v_{k}\in V$ such that $k=\pi(k)v_{k}$.
We have seen that $\pi(K)$ commutes with $H$ as well as with $V$.
Hence for any $h\in H$ we have that 
\[
[k,h]=[\pi(k)v_{k},h]=[v_{k},h]=v_{k}-\rho(h)v_{k}.
\]
On the other hand, $[k,h]\in K$ (since $K$ is normal) and $[v_{k},h]\in V$
(since $V$ is normal), so this element belongs to $K\cap V$ and
is thus trivial. We see that $v_{k}$ is fixed by all of $H$. We
know that $\overline{\pi(K)}$ acts trivially on all of $V$, so $v_{k}$
is fixed by $H$ as well as by $\overline{\pi(K)}$, hence by all
of $L$. But by definition of Abels--Noskov groups, the only vector
in $V$ fixed by $L$ is the zero vector. This means that $v_{k}=0$,
and so $k=\pi(k)\in\ker\rho\leqslant L$.
\end{proof}
We are now ready to complete the case $\ker f\cap V$ is trivial.
Observe that, since $\rho$ admits no non-zero fixed vectors and $S$
is perfect, a direct computation shows that the commutator subgroup
of $(S\times A)\ltimes V$ is $S\ltimes V$.
\begin{lem}
\label{lem:trivial-case}Let $G=(S\times A)\ltimes_{\rho}V$ be an
Abels--Noskov group, and let $f:G\to H$ be an open epimorphism
such that $\ker f\cap V$ is trivial. Let $f^{\mathrm{ab}}:A\to H/f(S\ltimes V)$
be the map induced by $f$. If $h_{1},\dots,h_{n}\in H$ are generators
with $n\geqslant\tr G$, then they admit generating lifts to $G$
via $f$ if and only if $h_{1}f(S\ltimes V),\dots,h_{n}f(S\ltimes V)$
admit generating lifts to $A$ via $f^{\mathrm{ab}}$. If $H$ is
Hausdorff, then so is $H/f(S\ltimes V)$.
\end{lem}

\begin{proof}
It is clear that, if $h_{1},\dots,h_{n}$ admit generating lifts
via $f$, then $h_{1}f(S\ltimes V),\dots,h_{n}f(S\ltimes V)$ admit
generating lifts to $A$ via $f^{\mathrm{ab}}$. Thus, suppose $h_{1},\dots,h_{n}\in H$
are generators (with $n\geqslant\tr G$) such that $h_{1}f(S\ltimes V),\dots,h_{n}f(S\ltimes V)$
admit generating lifts to $A$ via $f^{\mathrm{ab}}$. Set $L=S\times A$
and let $(\ell_{1},v_{1}),\dots,(\ell_{n},v_{n})\in L\ltimes V$ be
arbitrary lifts of $h_{1},\dots,h_{n}$. Set $\Delta=\ker f$ and
\[
F=\overline{\left\langle (\ell_{1},v_{1}),\dots,(\ell_{n},v_{n})\right\rangle },
\]
so that $F\Delta$ is dense in $G$. What we need to show is that
there are $\delta_{1},\dots,\delta_{n}\in\Delta$ such that $(\ell_{1},v_{1})\delta_{1},\dots,(\ell_{n},v_{n})\delta_{n}$
generate $G$. 

By the previous lemma, we know $\Delta$ is contained in $L$ and
acts trivially on $V$. Set $\bar{L}=L/\overline{\Delta}$, so that
$G/\overline{\Delta}$ is naturally isomorphic with $\bar{G}\coloneqq\bar{L}\ltimes_{\bar{\rho}}V$,
where $\bar{\rho}:\bar{L}\to\mathrm{GL}(V)$ is the representation
induced by $\rho$. Then $\bar{G}$ too is Abels--Noskov. Let $\bar{\ell}_{i}$
be the image of $\ell_{i}$ in $\bar{L}$. By assumption, $(\bar{\ell}_{1},v_{1}),\dots,(\bar{\ell}_{n},v_{n})$
generate $\bar{G}$, so
\[
\alpha_{\boldsymbol{v}}(\ker\varepsilon_{\boldsymbol{\bar{\ell}}})=V
\]
by \thmref{gen-of-AN}, where $\alpha_{\boldsymbol{v}},\varepsilon_{\bar{\boldsymbol{\ell}}}$
are defined as usual.

Observe that, since $\overline{\Delta}$ acts trivially on $V$, we
have that $\rho(L)=\bar{\rho}(\bar{L})$, so that $\mathbb{R}[\rho]=\mathbb{R}[\bar{\rho}]$
(where this is a genuine equality of subsets of $\mathrm{End}_{\mathbb{R}}(V)$,
not just a natural isomorphism). This means that $\varepsilon_{\boldsymbol{\bar{\ell}}}=\varepsilon_{\boldsymbol{\ell}}$
and hence $\ker\varepsilon_{\boldsymbol{\bar{\ell}}}=\ker\varepsilon_{\boldsymbol{\ell}}$,
so that 
\[
\alpha_{\boldsymbol{v}}(\ker\varepsilon_{\boldsymbol{\ell}})=\alpha_{\boldsymbol{v}}(\ker\varepsilon_{\boldsymbol{\bar{\ell}}})=V.
\]

Observe that $f:L\ltimes V\to H$ and $L\longrightarrow\bar{L}$ give
rise to the same abelianisation map $f^{\mathrm{ab}}:A\to H/f(S\ltimes V)$.
Thus, by \propref{Gas-Red}, there are $\delta_{1},\dots,\delta_{n}$
such that $\ell_{1}\delta_{1},\dots,\ell_{n}\delta_{n}$ generate
$L$. Since $\Delta$ acts trivially on $V$, we have that $\varepsilon_{\boldsymbol{\ell\delta}}=\varepsilon_{\boldsymbol{\ell}}$.
Therefore, by \thmref{gen-of-AN}, the elements $(\ell_{1}\delta_{1},v_{1}),\dots,(\ell_{n}\delta_{n},v_{n})$
generate $G$, as needed.

If $H$ is Hausdorff, then $\ker f\subseteq L$ is closed, so $\ker f^{\mathrm{ab}}$
is closed as well (by \propref{Gas-Red}).
\end{proof}

\subsection{The General Case}
\begin{prop}
\label{prop:AN-case}Let $G=(S\times A)\ltimes_{\rho}V$ be an Abels--Noskov
group, and let $f:G\to H$ be an open epimorphism. Let $f^{\mathrm{ab}}:A\to H/f(S\ltimes V)$
be the induced epimorphism. If $h_{1},\dots,h_{n}\in H$ are generators
with $n\geqslant\tr G$, then they admit generating lifts to $G$
via $f$ if and only if $h_{1}f(S\ltimes V),\dots,h_{n}f(S\ltimes V)$
admit generating lifts to $A$ via $f^{\mathrm{ab}}$.

If $H$ is Hausdorff, then so is $H/f(S\ltimes V)$.
\end{prop}

\begin{proof}
Let $h_{1},\dots,h_{n}\in H$ be generators such that $h_{1}f(S\ltimes V),\dots,h_{n}f(S\ltimes V)$
admit generating lifts to $A$ via $f^{\mathrm{ab}}$. Denote $\Delta=\ker f$.
Let $g_{1},\dots,g_{n}\in G$ be arbitrary lifts of $h_{1},\dots,h_{n}$,
so that $\left\langle g_{1},\dots,g_{n}\right\rangle \Delta$ is dense
in $G$.

The idea of the proof is very simple: we know how to do this for kernels
contained in $V$, and we know how to do this for kernels intersecting
$V$ trivially; so we first divide by the intersection with $V$,
which is a submodule and hence a closed subgroup, and then we divide
by the rest, which intersects the image of $V$ trivially. We now
spell this out in detail.

Consider $U\coloneqq\Delta\cap V$. It is a normal subgroup of $G$
contained in $V$, so it is a submodule of $V$. In particular, it
is closed. Therefore, $G/U$ is naturally isomorphic with $L\ltimes_{\overline{\rho}}W$,
$W\subseteq V$ is some submodule complementing $U$, i.e.~such that
$V=U\oplus W$, and $\bar{\rho}=\rho\res_{W}$. 

Since $\left\langle g_{1},\dots,g_{n}\right\rangle \Delta$ is dense
in $G$, we get that $\left\langle g_{1}U,\dots,g_{n}U\right\rangle \hat{\Delta}$
is dense in $G/U$, where $\hat{\Delta}$ is the image of $\Delta$
in $G/U$. Clearly, $\hat{\Delta}\cap W$ is trivial. Now, $G/U=L\ltimes_{\overline{\rho}}W$
is still an Abels--Noskov group, and its abelianisation is $A$.
Moreover, the map $\bar{f}:G/U\to(G/U)/\hat{\Delta}$ gives rise to
the same map $f^{\mathrm{ab}}:A\to H/f(S\ltimes V)$. Thus, since
$h_{1}f(S\ltimes V),\dots,h_{n}f(S\ltimes V)$ admit generating lifts
to $A$, we may use \lemref{trivial-case} and get that there are
$\bar{\varepsilon}_{1},\dots,\bar{\varepsilon}_{n}\in\hat{\Delta}$
such that $g_{1}U\cdot\bar{\varepsilon}_{1},\dots,g_{n}U\cdot\bar{\varepsilon}_{n}$
generate $G/U$. This means that there are $\varepsilon_{1},\dots,\varepsilon_{n}\in\Delta$
such that $g_{1}\varepsilon_{1}U,\dots,g_{n}\varepsilon_{n}U$ generate
$G/U$. Therefore, by \lemref{submodule-case} (since $n\geqslant\tr G\geqslant\mathrm{d}_{\mathbb{R}[\rho]}(V)+1$),
there are $u_{1},\dots,u_{n}\in U\subseteq\Delta$ such that $g_{1}\varepsilon_{1}u_{1},\dots,g_{n}\varepsilon_{n}u_{n}$
generate $G$, as needed.

If $H$ is Hausdorff, then $H/f(S\ltimes V)$ is Hausdorff as well
by \lemref{trivial-case} (applied to $G/U$, which is Hausdorff in
any case).
\end{proof}

\section{General Connected Lie Groups\label{sec:GeneralLie}}
\begin{thm}[{\cite[Corollary 5.8]{AN24}}]
Let $G$ be a connected Lie group. Then it admits a non-generating
subgroup $B\trianglelefteqslant G$, an Abels--Noskov group $L\ltimes_{\rho}V$
and a finite covering map $L\ltimes_{\rho}V\to G/B$.
\end{thm}

\begin{lem}
Let $G$ be a connected Lie group. Let $B$ be a non-generating subgroup
$B\trianglelefteqslant G$ such that there is an Abels--Noskov group
$(S\times A)\ltimes_{\rho}V$ and a finite covering map $(S\times A)\ltimes_{\rho}V\longrightarrow G/B$.
Then $(G/B)'$ is a closed subgroup, and the induced maps 
\[
G/\overline{G'}\longrightarrow(G/B)/(G/B)'\longleftarrow A
\]
are finite coverings. In particular, $A$ is isomorphic to $G/\overline{G'}$.
\end{lem}

\begin{proof}
Since $\rho$ admits no nonzero fixed vectors and $S$ is perfect,
a straightforward computations shows that the commutator subgroup
of $(S\times A)\ltimes_{\rho}V$ is $S\ltimes_{\rho}V$. In particular,
it is closed, so its image in $G/B$ is closed, since the map $(S\times A)\ltimes_{\rho}V\longrightarrow G/B$
is a closed map (being a finite covering). But its image is exactly
$(G/B)'$ (being an epimorphism), so $(G/B)'$ is closed. It is easy
to see in general that, if $f:H_{1}\to H_{2}$ is a finite covering,
then $(H_{1}/H_{1}')\longrightarrow(H_{2}/H_{2}')$ is a finite covering,
so $A\longrightarrow(G/B)/(G/B)'$ is a finite covering as well. 

The image of $G'$ in $G/B$ is equal to $(G/B)'$, which is closed,
so it is equal to the image of $\overline{G'}$ in $G/B$. The fact
the image of $\overline{G'}$ is closed in $G/B$ means that $B\overline{G'}$
is closed in $G$, which means that the image of $B$ in $G/\overline{G'}$
is closed. Now, by \cite[Lemma 2.4]{AN24}, the image of $B$ in $G/\overline{G'}$
is non-generating, so it is a torsion subgroup; being closed, it must
be finite. Since it is the kernel of the map $G/\overline{G'}\longrightarrow(G/B)/(G/B)'$,
we get that this map is a finite covering.
\end{proof}
\begin{thm}
\label{thm:main-thm}Let $G$ be a connected Lie group, $f:G\to H$
an open epimorphism. Let $f^{\mathrm{ab}}:G/\overline{G'}\to H/f\left(\overline{G'}\right)$
be the map induced by $f$. If $h_{1},\dots,h_{n}\in H$ are generators
with $n\geqslant\tr G$, then they admit generating lifts to $G$
via $f$ if and only if $h_{1}f\left(\overline{G'}\right),\dots,h_{n}f\left(\overline{G'}\right)$
admit generating lifts to $G/\overline{G'}$ via $f^{\mathrm{ab}}$. 

Thus, if in addition $\ker f$ is topologically-finitely-generated\footnote{Recall that $\ker f$ is automatically topologically-finitely-generated
if $H$ is Hausdorff.} and $n\geqslant2\dim(G/\overline{G'})-\dim T$ (where $T$ is the
maximal torus of the abelianisation $G/\overline{G'}$), then $h_{1},\dots,h_{n}$
admit generating lifts.
\end{thm}

\begin{rem}
In this case, it is not necessarily true that, if $H$ is Hausdorff,
then $H/f\left(\overline{G'}\right)$ must be Hausdorff as well.
\end{rem}

\begin{proof}
Fix a non-generating subgroup $B\trianglelefteqslant G$ such that
there is an Abels--Noskov group $(S\times A)\ltimes_{\rho}V$ and
a finite covering map $\pi:(S\times A)\ltimes_{\rho}V\to G/B$. Set
$\Delta=\ker f$, and consider the following commutative diagram:
\[
\xymatrix{G\ar[dddd]_{\vartheta}\ar[rd]^{\alpha}\ar[rr]^{f} &  & H\ar[dr]^{\bar{\alpha}}\ar[dddd]^{\bar{\vartheta}}\\
 & G/\overline{G'}\ar[dddd]_{q}\ar[rr]^{f^{\mathrm{ab}}} &  & H/f\left(\overline{G'}\right)\ar[dddd]\\
\\\\G/B\ar[rd]^{\beta}\ar[rr]^{\hat{f}} &  & H/f(B)\ar[dr]^{\bar{\beta}}\\
 & (G/B)/(G/B)'\ar[rr]^{\hat{f}^{\mathrm{ab}}} &  & (H/f(B))/\hat{f}((G/B)')\\
\\\\(S\times A)\ltimes_{\rho}V\ar[uuuu]^{\pi}\ar[rd]^{\gamma}\ar[rr]^{\tilde{f}} &  & \left((S\times A)\ltimes_{\rho}V\right)/\pi^{-1}(\vartheta(\Delta))\ar[uuuu]^{\bar{\pi}}\ar[dr]^{\bar{\gamma}}\\
 & A\ar[uuuu]^{p}\ar[rr]^{\tilde{f}^{\mathrm{ab}}} &  & A/\gamma(\pi^{-1}(\vartheta(\Delta)))\ar[uuuu]
}
\]
Observe that, even if $H$ is Hausdorff, it is possible that some
of the groups in the diagram are not (e.g., $H/f\left(\overline{G'}\right)$).

The maps $\alpha,\beta,\gamma$ are the topological abelianisation
maps -- the quotient by the closure of the commutator subgroup (which
happen to be equal to the abstract commutator subgroup in the case
of the latter two). The corresponding maps to their right, $\bar{\alpha},\bar{\beta},\bar{\gamma}$,
are abelianisation maps of sorts: quotients by subgroups that are
somewhere between the abstract commutator subgroup and its closure.

The leftmost vertical maps, $\pi,p,q,\vartheta$, are all absolutely
Gasch\"utz (as are in fact the rightmost vertical maps); the maps $\pi,p,q$
are even finite covering maps. We therefore get that the maps $f,\hat{f},\tilde{f},\tilde{f}^{\mathrm{ab}},\hat{f}^{\mathrm{ab}},f^{\mathrm{ab}}$
are all `Gasch\"utz equivalent', so $h_{1},\dots,h_{n}\in H$ admit
generating lifts via $f$ if and only if $h_{1}f\left(\overline{G'}\right),\dots,h_{n}f\left(\overline{G'}\right)$
admit generating lifts via $f^{\mathrm{ab}}$. 

To spell it out: let $h_{1},\dots,h_{n}\in H$ be generators. Since
$\vartheta$ is absolutely Gasch\"utz, they admit generating lifts to
$G$ via $f$ if and only if $h_{1}f(B),\dots,h_{n}f(B)$ admit generating
lifts to $G/B$ via $\hat{f}$. Take some arbitrary lifts $x_{1},\dots,x_{n}\in\left((S\times A)\ltimes_{\rho}V\right)/\pi^{-1}(\vartheta(\Delta))$
of $h_{1}f(B),\dots,h_{n}f(B)$; since $\pi:(S\times A)\ltimes_{\rho}V\to G/B$
is absolutely Gasch\"utz, we get that $x_{1},\dots,x_{n}$ admit generating
lifts to $(S\times A)\ltimes_{\rho}V$ via $\tilde{f}$ if and only
if $h_{1}f(B),\dots,h_{n}f(B)$ admit generating lifts to $G/B$ via
$\hat{f}$. Now, denote the images of $x_{1},\dots,x_{n}$ in $A/\gamma(\pi^{-1}(\vartheta(\Delta)))$
by $\bar{x}_{1},\dots,\bar{x}_{n}$; by \propref{AN-case}, $x_{1},\dots,x_{n}$
admit generating lifts to $(S\times A)\ltimes_{\rho}V$ via $\tilde{f}$
if and only if $\bar{x}_{1},\dots,\bar{x}_{n}$ admit generating lifts
to $A$ via $\tilde{f}^{\mathrm{ab}}$. Denote the images of $\bar{x}_{1},\dots,\bar{x}_{n}$
in $(H/f(B))/\hat{f}((G/B)')$ by $y_{1},\dots,y_{n}$. Then, since
$p$ is absolutely Gasch\"utz, $\bar{x}_{1},\dots,\bar{x}_{n}$ admit
generating lifts to $A$ via $\tilde{f}^{\mathrm{ab}}$ if and only
if $y_{1},\dots,y_{n}$ admit generating lifts to $G/B$ via $\hat{f}^{\mathrm{ab}}$.
By the commutativity of the diagram, $y_{1},\dots,y_{n}$ are the
images of $h_{1}f\left(\overline{G'}\right),\dots,h_{n}f\left(\overline{G'}\right)$
under the corresponding map, so, since $q$ is absolutely Gasch\"utz,
$y_{1},\dots,y_{n}$ admit generating lifts via $\hat{f}^{\mathrm{ab}}$
if and only if $h_{1}f\left(\overline{G'}\right),\dots,h_{n}f\left(\overline{G'}\right)$
admit generating lifts via $f^{\mathrm{ab}}$. 
\end{proof}
This completes the proof of Theorem \ref{thm:intro-main}, and of
its immediate consequence Theorem \ref{thm:intro-perfect-groups}.

\section{The Gasch\"utz Rank\label{sec:GasRank}}

Recall \defref{gaschutz-rank} of the Gasch\"utz rank of a group. Since
we reduced the problem of lifting generators to epimorphisms between
abelian groups (\thmref{intro-main}), and computed the Gasch\"utz rank
of connected abelian Lie groups (\thmref{into-abel}), it might seem
like we should immediately obtain a computation for the Gasch\"utz rank
of any connected Lie group. This problem, however, turns out to be
a little subtler than it first appears.

Let $f:G\to H$ be an epimorphism between connected Lie groups, and
let $f^{\mathrm{ab}}:G/\overline{G'}\to H/f(\overline{G'})$ be the
map induced by $f$. \thmref{intro-main} says that, if $h_{1},\dots,h_{n}\in H$
generate $H$, then they can be lifted to generators of $G$ through
$f$ if and only if $\bar{h}_{1},\dots,\bar{h}_{n}\in H/f(\overline{G'})$
(the images of $h_{1},\dots,h_{n}$ in $H/f(\overline{G'})$) can
be lifted to $G/\overline{G'}$ through $f^{\mathrm{ab}}$. The subtlety
lies in the fact $\ker f^{\mathrm{ab}}$ might not be closed. This
is why, in \propref{AbelianCase}, we had to allow not-necessarily-closed
subgroups (equivalently, non-Hausdorff target groups). This gives
us an upper bound on the Gasch\"utz rank of $G$. Namely, if $G$ is
any connected Lie group, then we have shown that

\[
\tr G\leqslant\Gas G\leqslant\max\left\{ \tr G,2\dim G/\overline{G'}-\dim T\right\} ,
\]
where $T$ is the maximal torus of $G/\overline{G'}$. If $G/\overline{G'}$
is non-compact, then $\Gas{G/\overline{G'}}=2\dim G/\overline{G'}-\dim T$,
which means that $\Gas G$ is actually equal to $\max\left\{ \tr G,2\dim G/\overline{G'}-\dim T\right\} $
(since the Gasch\"utz rank cannot increase in quotients). However, if
$G/\overline{G'}$ is compact, then $\Gas{G/\overline{G'}}=1$ (e.g.,
by \cite{CG18}), so $\Gas G$ might be strictly smaller than $\max\left\{ \tr G,2\dim G/\overline{G'}-\dim T\right\} $.

In most cases, however, the Gasch\"utz rank can be computed by our results:
\begin{enumerate}
\item If $G$ is compact, then we have $\Gas G=\tr G$ (by \cite{CG18}).
\item \label{enu:noncompact}If $G/\overline{G'}$ is non-compact, then
we have $\Gas G=\max\left\{ \tr G,2\dim G/\overline{G'}-\dim T\right\} $.
\item If $\tr G\geqslant2\dim G/\overline{G'}-\dim T$ (for example, if
$G$ is perfect), then $\Gas G=\tr G$.
\item If every Hausdorff quotient of $G$ satisfies the condition that its
abstract commutator subgroup is closed,\footnote{For instance, every Abels--Noskov group satisfies this condition
(and hence every group which is Abels--Noskov up to finite coverings),
by \propref{AN-case}.} then $\Gas G=\max\{\tr G,\Gas{G/G'}\}$. This is because, in this
case, $f(G')=H'$ for every open epimorphism, which means that the
kernel of $f^{\mathrm{ab}}:G/G'\to H/H'$ is closed. 
\end{enumerate}
We are left with the case $G$ is non-compact, $G/\overline{G'}$
is a torus with $\dim(G/\overline{G'})>\tr G$, and $f:G\to H$ is
a an open epimorphism onto a group $H$ such that $H'$ is not closed.

\end{document}